\documentclass{article}

\usepackage{amsfonts}
\usepackage{amsmath}
\usepackage{amssymb}
\usepackage{amsthm}
\usepackage{bookmark}
\usepackage{caption}
\usepackage{cleveref}
\usepackage{enumitem}
\usepackage{float}
\usepackage{galois}
\usepackage[a4paper]{geometry}
\usepackage{hyperref}
\usepackage{indentfirst}
\usepackage{mathrsfs}
\usepackage{mathtools}
\usepackage{multirow}
\usepackage{subfiles}
\usepackage{tasks}
\usepackage{thmtools}
\usepackage{tikz-cd}

\def\afs{\mathbb{A}}
\def\bad{\mathrm{bad}}
\def\badpair{\mathrm{BadPair}}
\def\coor{\mathrm{Coor}}
\def\diag{\mathrm{diag}}
\def\edge{\mathrm{E}}

\def\ex{\mathrm{ex}}
\def\frf{K}
\def\gal{\mathrm{Gal}}
\def\gauss{\mathrm{G}}
\def\gkz{\mathscr{H}}
\def\gl{\mathrm{GL}}
\def\hyp{\mathscr{H}\mathrm{yp}}
\def\init{\operatorname{in}}
\def\intr{R}

\def\kum{\mathscr{K}}
\def\lbox{{\text{\large$\square$}}}
\def\mat{\mathrm{Mat}}
\def\mono{\mathrm{Mono}}
\def\orb{\mathrm{orb}}
\def\perm{\gamma}
\def\proj{\operatorname{Proj}}

\def\prs{\mathbb{P}}
\def\ql{\overline{\mathbb{Q}}_{\ell}}
\def\qtup{\mathrm{Quad}}
\def\rank{\operatorname{rank}}
\def\rev{\mathrm{rev}}
\def\rsf{\kappa}
\def\spec{\operatorname{Spec}}
\def\spmap{\operatorname{sp}}
\def\topo{\mathrm{top}}
\def\trs{\mathbb{T}}
\def\trv{\mathbb{A}}

\declaretheorem[name=Theorem,numberwithin=subsection]{thm}

\declaretheorem[name=Corollary,sibling=thm]{cor}
\declaretheorem[name=Definition,sibling=thm]{dfn}
\declaretheorem[name=Example,sibling=thm,style=remark]{exm}
\declaretheorem[name=Lemma,sibling=thm]{lem}
\declaretheorem[name=Proposition,sibling=thm]{prp}
\declaretheorem[name=Remark,sibling=thm,style=remark]{rem}

\crefformat{equation}{(#2#1#3)}
\crefformat{subsection}{subsection #2#1#3}

\numberwithin{equation}{subsection}

\title{The characteristic cycle of a non-confluent \texorpdfstring{$\ell$}{l}-adic GKZ hypergeometric sheaf}
\author{Peijiang Liu}

\begin{document}

\maketitle

\begin{abstract}
	An $\ell$-adic GKZ hypergeometric sheaf is defined analogously to a GKZ hypergeometric $\mathcal{D}$-module. We introduce an algorithm of computing the characteristic cycle of an $\ell$-adic GKZ hypergeometric sheaf of certain type. Our strategy is to apply a formula of the characteristic cycle of the direct image of an $\ell$-adic sheaf. We verify the requirements for the formula to hold by calculating the dimension of the direct image of a certain closed conical subset of cotangent bundle. We also define an $\ell$-adic GKZ-type sheaf whose specialization tensored with a constant sheaf is isomorphic to an $\ell$-adic non-confluent GKZ hypergeometric sheaf. On the other hand, the topological model of an $\ell$-adic GKZ-type sheaf is isomorphic to the image by the de Rham functor of a non-confluent GKZ hypergeometric $\mathcal{D}$-module whose characteristic cycle has been calculated. This gives an easier way to determine the characteristic cycle of an $\ell$-adic non-confluent GKZ hypergeometric sheaf of certain type.
\end{abstract}

\tableofcontents

\numberwithin{thm}{section}
\numberwithin{con}{section}
\numberwithin{cor}{section}
\numberwithin{dfn}{section}
\numberwithin{exm}{section}
\numberwithin{lem}{section}
\numberwithin{prp}{section}
\numberwithin{rem}{section}

\section{Introduction}
An $\ell$-adic GKZ hypergeometric sheaf is an \'etale sheaf over a finite field $\rsf$ of characteristic $p$ where $\ell\neq p$ is a prime number (cf. {\cite[Introduction]{fu2016l}}). $\ell$-adic GKZ hypergeometric sheaves can be viewed as a generalization of $\ell$-adic Kloosterman sheaves (cf. {\cite[Chapter~4]{katz1988gauss}}) and the $\ell$-adic hypergeometric sheaves (cf. {\cite[Chapter~8]{katz1990exponential}}). The function given by taking traces of geometric Frobenius endomorphism of an $\ell$-adic GKZ hypergeometric sheaf is equal to the corresponding hypergeomtric function over $\rsf$ which is also called a GKZ hypergeometric sum (cf. {\cite{gelfand2001hypergeometric}}). A GKZ hypergeometric sum is an analogue over $\rsf$ of the integral representation of solutions of a system of linear differential equations which is called a GKZ hypergeometric system.

In this article, we introduce a method to relate a non-confluent $\ell$-adic GKZ hypergeometric sheaf to a non-confluent GKZ hypergeometric $\mathcal{D}$-module, and then formulate an algorithm to compute the characteristic cycle of the former $\ell$-adic sheaf by the known result of the characteristic cycle of the latter $\mathcal{D}$-module under certain assumptions.

We denote by $\hyp_{\psi}(A,\chi)\in D_{c}^{b}(\afs^{n},\overline{\mathbb{Q}}_{\ell})$ the $\ell$-adic GKZ hypergeometric sheaf associated to a matrix $A$ of dimension $(d+1)\times n$ whose entries are integers, a multiplicative character $\chi:(\rsf^{*})^{d+1}\rightarrow\overline{\mathbb{Q}}_{\ell}^{*}$ and a nontrivial additive character $\psi:\rsf\rightarrow\overline{\mathbb{Q}}_{\ell}^{*}$. We say that $A$ is \textit{non-confluent} if there exists an invertible matrix $P$, such that every element of the first row of $PA$ is equal to $1$. Let $R$ be a discrete valuation ring whose residue field is $\rsf$, such that the field of fractions $\frf$ of $R$ is a finite extension of $\mathbb{Q}$. For a scheme $Y$ over $R$, we denote by $Y_{\Lambda}=Y\times_{\spec R}\spec\Lambda$ where $\Lambda$ denotes one of $\rsf$ and $\frf$. For a morphism $f:Y'\rightarrow Y$ of schemes over $R$, we denote by $f_{\Lambda}$ the morphism induced by the base change $\spec\Lambda\rightarrow\spec R$. For a sheaf $\mathscr{F}\in D_{c}^{b}(Y,\overline{\mathbb{Q}}_{\ell})$, we denote by $\mathscr{F}_{\Lambda}$ the inverse image along $Y_{\Lambda}\rightarrow Y$. Assume that $A$ is non-confluent, then there exist a scheme $T$ over $R$ and $\mathscr{F}_{\chi}(A)\in D_{c}^{b}(T\times\afs^{n}_{R},\overline{\mathbb{Q}}_{\ell})$ which is tamely ramified along the complement of its lisse locus $U_{A}\subset T\times\afs^{n}_{R}$, such that
\begin{equation*}
	\hyp_{\psi}(A,\chi)\cong R\pi_{\rsf !}\mathscr{F}_{\chi}(A)_{\rsf}\otimes G(\chi,\psi),
\end{equation*}
where $G(\chi,\psi)$ is a $\overline{\mathbb{Q}}_{\ell}$-vector space and $\pi:T\times\afs^{n}\rightarrow\afs^{n}$ is the projection. Next, we say that $A$ is \textit{$p$-nondegenerate} if
\begin{equation*}
	\rank A[\theta]=\rank_{\rsf}A[\theta]_{\rsf}
\end{equation*}
holds for any subset $\theta\subset\{1,\dots,n\}$, where $A[\theta]$ denotes the matrix consisting of the columns of $A$ with indices in $\theta$ and $A[\theta]_{\rsf}$ is the image of $A[\theta]$ by $R\rightarrow\rsf$. For a scheme $Y$ over $R$, We denote by
\begin{equation*}
	\spmap:CH_{n}(Y_{\frf}/\frf)\rightarrow CH_{n}(Y_{\rsf}/\rsf)
\end{equation*}
the specialization map. Then, the main result is given by the following theorem.

\begin{thm}\label{introduction-main-result}
	Let $A$ be a non-confluent matrix. Let $\chi:(\rsf^{*})^{d+1}\rightarrow\overline{\mathbb{Q}}_{\ell}^{*}$ be a multiplicative character, and let $\psi:\rsf\rightarrow\mathbb{Q}_{\ell}^{*}$ be a nontrivial additive character. We denote by $j:\afs^{n}\rightarrow\prs^{n}$ the open immersion.
	\begin{enumerate}
		\item\label{introduction-main-result-nondegenerate} If $A$ is $p$-nondegenerate, then there exists an embedding $\widetilde{j}_{A}:U_{A}\rightarrow\widetilde{T}_{A}\times\prs^{n}$, such that
		\begin{equation}\label{introduction-main-formula}
			CC\hyp_{\psi}(A,\chi)=\spmap CCR\pi_{\frf !}\mathscr{F}_{\chi}(A)_{\frf}\in Z_{n}(T^{*}\afs^{n}_{\rsf}/\rsf).
		\end{equation}
		
		\item\label{introduction-main-result-square} If $A$ is a square matrix (i.e. $d+1=n$), then there exist a non-confluent matrix $A'$ which is $p$-nondegenerate, and a multiplicative character $\chi':(\rsf^{*})^{d+1}\rightarrow\overline{\mathbb{Q}}_{\ell}^{*}$, such that
		\begin{equation*}
			\hyp_{\psi}(A,\chi)\cong\hyp(A',\chi').
		\end{equation*}
	\end{enumerate}
\end{thm}

For a more precise form of \Cref{introduction-main-result}.\ref{introduction-main-result-nondegenerate}, see \Cref{square-characteristic-cycle}, and for a more precise form of \Cref{introduction-main-result}.\ref{introduction-main-result-square}, see \Cref{square-p-nondegerate-non-confluent}. We also prove the following theorem so that \cref{introduction-main-formula} holds when $A$ is a non-confluent square matrix. Next, we explain the idea of the proof of \Cref{introduction-main-formula}. Our ingredient is the following theorem.

\begin{thm}[{\cite[Theorem~2.2.5]{saito2021characteristic}}]\label{introduction-direct-image-formula}
	Let $f:Y'\rightarrow Y$ be a morphism of smooth schemes over a perfect field. Let $\mathscr{F}\in D_{c}^{b}(Y',\overline{\mathbb{Q}}_{\ell})$, then we have
	\begin{equation}\label{introduction-direct-image-formula-cycle}
		CCRf_{*}\mathscr{F}=f_{!}CC\mathscr{F}\in Z_{\dim Y}(T^{*}Y)
	\end{equation}
	if the following conditions are satisfied.
	\begin{enumerate}
		\item $Y$ is projective.
		\item $f$ is quasi-projective as well as proper.
		\item $\dim f_{\circ}SS\mathscr{F}\leq\dim Y$.
	\end{enumerate}
\end{thm}

\begin{rem}
	A stronger version of \Cref{introduction-direct-image-formula} is recently proved in {\cite{abe2022ramification}} where \cref{introduction-direct-image-formula-cycle} is claimed to hold even if $Y$ is not projective. Moreover, even if we remove the condition $\dim f_{\circ}SS\mathscr{F}\leq\dim Y$, the formula \cref{introduction-direct-image-formula-cycle} remain to hold in $CH_{\dim Y}(f_{\circ}SS\mathscr{F})[1/p]$.
\end{rem}

In order to apply \Cref{introduction-direct-image-formula}, we construct a commutative diagram
\begin{equation*}
\begin{tikzpicture}
	\draw(0,0)node{$\afs^{n}$};
	\draw(3,0)node{$\prs^{n}$};
	\draw(-3,1.5)node{$T\times\afs^{n}$};
	\draw(0,1.5)node{$\overline{T}\times\afs^{n}$};
	\draw(3,1.5)node{$\overline{T}\times\prs^{n}$};
	\draw(-6,1.5)node{$U_{A}$};
	\draw(1.5,3.5)node{$\widetilde{T}_{A}\times\afs^{n}$};
	\draw(4.5,3.5)node{$\widetilde{T}_{A}\times\prs^{n}$};
	%%%%%%
	\draw[->](2.25,3.5)--(3.75,3.5);
	\draw[rounded corners,->](0.75,3.25)--(0,2.75)--(0,1.75);
	\draw[rounded corners,->](3.75,3.25)--(3,2.75)--(3,1.75);
	\draw[rounded corners,->](4.5,3.25)--(4.5,0.75)--(3.25,0);\draw(4.75,2)node{\text{\scriptsize$\widetilde{\pi}_{A}$}};
	\draw[->](-5.5,1.5)--(-3.75,1.5);\draw(-4.5,1.75)node{\text{\scriptsize$j_{A}$}};
	\draw[rounded corners](-5.75,1.75)--(-5,2.25)--(-0.1,2.25);\draw(0.1,2.25)--(2.9,2.25);\draw[rounded corners,->](3.1,2.25)--(4.25,2.25)--(4.25,3.25);\draw(-1.5,2.5)node{\text{\scriptsize$\widetilde{j}_{A}$}};
	\draw[->](-2.25,1.5)--(-0.75,1.5);
	\draw[rounded corners,->](-3,1.25)--(-3,0)--(-0.5,0);\draw(-1.5,0.25)node{\text{\scriptsize$\pi$}};
	\draw[->](0.75,1.5)--(2.25,1.5);
	\draw[->](0,1.25)--(0,0.25);
	\draw[->](3,1.25)--(3,0.25);
	\draw[->](0.5,0)--(2.5,0);\draw(1.5,0.25)node{\text{\scriptsize$j$}};
	%%%%%%
	\draw(1.35,0.6)rectangle(1.65,0.9);
	\draw(1.343,2.6)--(1.657,2.6);\draw(1.51,2.896)--(1.81,2.896);
	\draw(1.35,2.6)--(1.35,2.7);\draw(1.65,2.6)--(1.65,2.7);
	\draw(1.35,2.7)arc(180:100:0.2);\draw(1.65,2.7)arc(180:100:0.2);
\end{tikzpicture}
\end{equation*}
where $U_{A}$ is the lisse locus of $\mathscr{F}_{\chi}(A)$, such that the following conditions are satisfied.
\begin{enumerate}
	\item $\overline{T}$ is projective and $\widetilde{T}_{A}\times\prs^{n}\rightarrow\overline{T}\times\prs^{n}$ is smooth as well as projective.
	\item $\widetilde{j}_{A}$ is an open immersion and $(\widetilde{T}_{A}\times\prs^{n})\backslash U_{A}$ is a divisor with simple normal crossings.
	\item $\dim(\widetilde{\pi}_{A})_{\rsf\circ}SS(\widetilde{j}_{A})_{\rsf !}(j_{A})_{\rsf}^{*}\mathscr{F}_{\chi}(A)_{\rsf}\leq n$.
\end{enumerate}
If such diagram exists, then $(\widetilde{\pi}_{A})_{\rsf}:(\widetilde{T}_{A})_{\rsf}\times\prs^{n}_{\rsf}\rightarrow\prs^{n}_{\rsf}$ and $(\widetilde{j}_{A})_{\rsf !}(j_{A})_{\rsf}^{*}\mathscr{F}_{\chi}(A)_{\rsf}$ meet the requirements of \Cref{introduction-direct-image-formula} and hence
\begin{equation*}
	CCR(\widehat{\pi}_{A})_{\rsf !}(\widetilde{j}_{A})_{\rsf !}(j_{A})_{\rsf}^{*}\mathscr{F}_{\chi}(A)=(\widehat{\pi}_{A})_{\rsf !}CC(\widetilde{j}_{A})_{\rsf !}(j_{A})_{\rsf}^{*}\mathscr{F}_{\chi}(A)\in Z_{n}(T^{*}\prs^{n}_{\rsf}/\rsf),
\end{equation*}
where $CC(\widetilde{j}_{A})_{\rsf !}(j_{A})_{\rsf}^{*}\mathscr{F}_{\chi}(A)$ is given by the cycle class of the conormal bundle alone the divisor $(\widetilde{T}_{A})_{\rsf}\times\prs^{n}_{\rsf}$ with simple normal crossings. See \Cref{resolution-existence} for the construction of the resolution of singularities $\widetilde{T}_{A}\times\afs^{n}\rightarrow\overline{T}\times\afs^{n}$. In this situation, \Cref{dimension-global-dimension} implies that if $A$ is $p$-nondegenerate, then we have $\dim(\widetilde{\pi}_{A})_{\rsf\circ}SS(\widetilde{j}_{A})_{\rsf !}(j_{A})_{\rsf}^{*}\mathscr{F}_{\chi}(A)_{\rsf}\leq n$. On the other hand, by {\cite[Proposition~2.2.7]{saito2021characteristic}}, we also have $\dim(\widetilde{\pi}_{A})_{\frf\circ}SS(\widetilde{j}_{A})_{\frf !}(j_{A})_{\frf}^{*}\mathscr{F}_{\chi}(A)_{\frf}\leq n$. Since the specialization maps commute with the direct image map $(\widetilde{\pi}_{A})_{\Lambda !}$ and the flat pull-back map $j_{\Lambda}^{*}$, the equality \cref{introduction-main-formula} is proved.

However, the computation of the direct images of cycles can be very complicated in general. In order to avoid falling into chaos, we relate $\hyp_{\psi}(A,\chi)$ to a GKZ hypergeometric $\mathcal{D}$-module via Riemann-Hilbert correspondence and compare their characteristic cycles. The characteristic cycle of a GKZ hypergeometric $\mathcal{D}$-module is well studied so that there is a formula to compute it (cf. \cite{berkesch2020characteristic}). Then, the characteristic cycle of $\hyp_{\psi}(A,\chi)$ is given by the formula computing that of the related GKZ hypergeometric $\mathcal{D}$-module and the comparison between them.

We denote by $\mathcal{H}_{z}(A)$ the GKZ hypergeometric $\mathcal{D}$-module associated to $A$ and $z\in(\mathbb{C}^{*})^{d+1}$. If $A$ is non-confluent, then $\mathcal{H}_{z}(A)$ is regular holonomic and we prove that
\begin{equation*}
	DR(\mathcal{H}_{z}(A))\cong (R\pi_{\frf !}\mathscr{F}_{\chi}(A)_{\frf})_{\topo}\in D_{c}^{b}(\afs^{n}_{\topo},\mathbb{C}).
\end{equation*}
Here $\afs^{n}_{\topo}$ is the topological space of $\afs^{n}_{\mathbb{C}}$ and $(R\pi_{\frf !}\mathscr{F}_{\chi}(A)_{\frf})_{\topo}$ is the topological sheaf given by the inverse image of $R\pi_{\frf !}\mathscr{F}_{\chi}(A)_{\frf}$ by the base change $\spec\mathbb{C}\rightarrow\spec\Lambda$. Recall that $j_{A}^{*}\mathscr{F}_{\chi}$ is a lisse sheaf on $U_{A}$ and is tamely ramified along $\widetilde{T}_{A}\times\afs^{n}\setminus U_{A}$, which implies that
\begin{equation*}
	(CC(\widetilde{j}_{A})_{\frf !}(j_{A})_{\frf}^{*}\mathscr{F}_{\chi}(A)_{\frf})\times_{\spec\frf}\spec\mathbb{C}=CC((\widetilde{j}_{A})_{\frf !}(j_{A})_{\frf}^{*}\mathscr{F}_{\chi}(A)_{\frf})_{\topo}.
\end{equation*}
By {\cite[Proposition~9.4.2,~Proposition~9.4.3]{kashiwara2003d}}, we prove
\begin{equation}\label{introduction-topological-formula}
	CC\mathcal{H}_{z}(A)=(CCR\pi_{\frf !}\mathscr{F}_{\chi}(A)_{\frf})\times_{\spec\frf}\spec\mathbb{C}.
\end{equation}
Then, combining \cref{introduction-main-formula} with \cref{introduction-topological-formula}, we get a formula computing $CC\hyp_{\psi}(A,\chi)$. See \Cref{topology-comparison-hard} for the precise form of the formula.

We move on to introduce the organization of this article. In \Cref{sheaf}, we review the definition of $\ell$-adic GKZ hypergeometric sheaves and define $\ell$-adic GKZ-type sheaves as a generalization to sheaves over a field of characteristic zero. In \Cref{singularity}, we construct the resolution of singularities. We begin with interpreting the problem in the language of toric varieties in \Cref{divisor}, and then introduce the tool we use to construct the resolution in \Cref{blowup}. \Cref{resolution} is the construction. Next, we formulate the main result of this article in \Cref{characteristic}. We calculate the dimension of certain cycles in \Cref{dimension}, and prove them to be the direct images of the singular support in \Cref{image}. \Cref{square} is the proof of the main result. In \Cref{comparison}, we formulate comparison theorems among different types of characteristic cycles. We use specialization map to build a connection between the characteristic cycle of a non-confluent $\ell$-adic GKZ hypergeometric sheaf and the corresponding $\ell$-adic GKZ-type sheaf over a field of characteristic zero in \Cref{specialization}, and use Riemann-Hilbert correspondence to construct a relationship between an $\ell$-adic GKZ-type sheaf and the corresponding regular GKZ hypergeometric $\mathcal{D}$-module in \Cref{topology}. Finally, we give examples in \Cref{example} to explain the limitation of our method.

I appreciate the financial support from my parents, and I wish to thank my teacher, Tomoyuki Abe, for so much helpful advice. I also wish to thank the referee and the editor for offering helpful suggestions.

\section{$\ell$-adic Kummer-type and GKZ-type sheaves}\label{sheaf}
We review the definition of $\ell$-adic GKZ hypergeometric sheaves and a proposition of non-confluent ones, inspired by which we define the $\ell$-adic GKZ-type sheaves as a generalization of the non-confluent $\ell$-adic GKZ hypergeometric sheaves to sheaves over a finite extension of $\mathbb{Q}$.

Let $\rsf=\mathbb{F}_{q}$ be a finite field of characteristic $p$. Let $\zeta_{q-1}$ be a primitive $(q-1)$st root of unity, and let $\frf=\mathbb{Q}(\zeta_{q-1})$. Then, the composition of the $p$-adic valuation of $\mathbb{Q}$ and the norm function associated to the field extension $\frf/\mathbb{Q}$ defines a discrete valuation of $\frf$. Let $\intr$ be the union of $\{0\}$ and the set of all elements in $\frf$ with non-negative $p$-adic valuations. Then, we note that $\intr$ is a discrete valuation ring whose field of fractions of $\intr$ is $\frf$ and whose residue field is $\rsf$. We denote by $\Lambda$ one of $\intr,\rsf$ and $\frf$.

For $0\leq i\leq d$, let $\trs^{1}_{i}=\spec\Lambda[t_{i}^{\pm 1}]$. For $0\leq k\leq d$, we write $\trs^{d+1-k}=\trs^{1}_{k}\times\dots\times\trs^{1}_{d}$. Let $\prescript{k}{}\tau_{i}:\trs^{d+1-k}\rightarrow\trs^{1}_{i}$ be the projection. Let $\afs^{n}=\Lambda[x_{1},\dots,x_{n}]$. Let $\prescript{k}{}\tau:\afs^{n}_{\trs^{d+1-k}}\rightarrow\trs^{d+1-k}$ and $\prescript{k}{}\pi:\afs^{n}_{\trs^{d+1-k}}\rightarrow\afs^{n}$ be the projections. Let $\afs^{1}_{\dagger}=\spec\Lambda[y]$. For a matrix $M=[m_{ij}]_{k\leq i\leq d;1\leq j\leq n}$ whose entries are integers, we define $\prescript{k}{}g_{M}:\afs^{n}_{\trs^{d+1-k}}\rightarrow\afs^{1}_{\dagger}$ to be the morphism induced by
\begin{equation*}
	\Lambda[y]\rightarrow\Lambda[t_{k},\dots,t_{d}][x_{1},\dots,x_{n}]:y\mapsto\sum_{j=1}^{n}x_{j}\prod_{i=1}^{d+1-k}t_{i}^{m_{ij}}.
\end{equation*}
In this article $k$ is only set to be $0$ or $1$, and we write $\widehat{\&}_{\%}=\prescript{0}{}\&_{\%}$ and $\&_{\%}=\prescript{1}{}\&_{\%}$. Let $\ell\neq p$ be a prime number. For a nontrivial additive character $\psi:\rsf\rightarrow\overline{\mathbb{Q}}_{\ell}^{*}$, we denote by $\mathscr{L}_{\psi}\in D_{c}^{b}(\afs^{1}_{\dagger},\overline{\mathbb{Q}}_{\ell})$ the Artin-Schreier sheaf associated to $\psi$. Let $\mu_{q-1}$ be the group of $(q-1)$st roots of unity. For a multiplicative character $\rho:\mu_{q-1}\rightarrow\overline{\mathbb{Q}}_{\ell}^{*}$, we denote by $\kum_{\rho}\in D_{c}^{b}(\mathbb{G}_{m},\overline{\mathbb{Q}}_{\ell})$ the Kummer sheaf associated to $\rho$.

\begin{dfn}[\cite{fu2016l}]
	Set $\Lambda=\rsf$. Let $A=[a_{ij}]_{0\leq i\leq d;1\leq j\leq n}$ be a matrix of rank $d+1$ whose entries are integers. Let $\chi=(\chi_{0},\dots,\chi_{d}):\mu_{q-1}^{d+1}\rightarrow\overline{\mathbb{Q}}_{\ell}^{*}$ be a multiplicative character, and let $\psi:\rsf\rightarrow\overline{\mathbb{Q}}_{\ell}^{*}$ be a nontrivial additive character. The \textnormal{$\ell$-adic GKZ hypergeometric sheaf} associated to $A,\chi$ and $\psi$ is defined to be
	\begin{equation*}
		\hyp_{\psi}(A,\chi)=R\widehat{\pi}_{!}(\widehat{\tau}^{*}(\widehat{\tau}_{0}^{*}\kum_{\chi_{0}}\otimes\dots\otimes\widehat{\tau}_{d}^{*}\kum_{\chi_{d}})\otimes\widehat{g}_{A}^{*}\mathscr{L}_{\psi})[d+n+1]\in D_{c}^{b}(\afs^{n},\overline{\mathbb{Q}}_{\ell}).
	\end{equation*}
\end{dfn}

Let $\trs^{1}_{\dagger}=\spec\Lambda[y^{\pm 1}]$, and let $j_{\dagger}:\trs^{1}_{\dagger}\rightarrow\afs^{1}_{\dagger}$ be the open immersion. Then, for a matrix $B=[b_{ij}]_{1\leq i\leq d;1\leq j\leq n}$ whose entries are integers, we define an open subset $U_{B}\subset\afs^{n}_{\trs^{d}}$ by the following Cartesian diagram of schemes over $\Lambda$.
\begin{equation}\label{sheaf-kummer-sheaf-lisse-locus}
\begin{tikzcd}
	U_{B}\arrow[d,"g_{B}'"']\arrow[r,"j_{B}"]\arrow[dr,phantom,"\square"]&\afs^{n}_{\trs^{d}}\arrow[d,"g_{B}"]\\
	\trs^{1}_{\dagger}\arrow[r,"j_{\dagger}"]&\afs^{1}_{\dagger}
\end{tikzcd}.
\end{equation}

\begin{dfn}\label{sheaf-definition}
	Let $B=[b_{ij}]_{1\leq i\leq d;1\leq j\leq n}$ be a matrix whose entries are integers, and let $\chi=(\chi_{0},\dots,\chi_{d}):\mu_{q-1}^{d+1}\rightarrow\overline{\mathbb{Q}}_{\ell}^{*}$ be a multiplicative character.
	\begin{enumerate}
		\item We define the \textnormal{$\ell$-adic Kummer-type sheaf} associated to $B,\chi$ to be
		\begin{equation*}
			\kum_{\chi}(B)=j_{B}^{*}\tau^{*}(\tau_{1}^{*}\kum_{\chi_{1}}\otimes\dots\otimes\tau_{d}^{*}\kum_{\chi_{d}})\otimes g_{B}'^{*}\kum_{\chi_{0}^{-1}}\in D_{c}^{b}(U_{B},\overline{\mathbb{Q}}_{\ell}).
		\end{equation*}
		
		\item We define the \textnormal{$\ell$-adic GKZ-type sheaf} associated to $B,\chi$ to be
		\begin{equation*}
			\gkz_{\chi}(B)=R\pi_{!}j_{B!}\kum_{\chi}(B)\in D_{c}^{b}(\afs^{n},\overline{\mathbb{Q}}_{\ell}).
		\end{equation*}
	\end{enumerate}
\end{dfn}

\begin{dfn}
	Let $B=[b_{ij}]_{1\leq i\leq d;1\leq j\leq n}$ be a matrix whose entries are integers.
	\begin{enumerate}
		\item We define $\widehat{B}=[b_{ij}]_{0\leq i\leq d;1\leq j\leq n}$ where $b_{01}=\dots=b_{0n}=1$.
		\item We say that $B$ is \textnormal{sub-non-confluent} if $\widehat{B}$ is of rank $d+1$.
	\end{enumerate}
\end{dfn}

\begin{lem}\label{sheaf-isomorphism}
	Set $\Lambda=\rsf$. Let $B=[b_{ij}]_{1\leq i\leq d;1\leq j\leq n}$ be a sub-non-confluent matrix. Let $\chi=(\chi_{0},\dots,\chi_{d}):\mu_{q-1}^{d+1}\rightarrow\overline{\mathbb{Q}}_{\ell}^{*}$ be a multiplicative character, and let $\psi:\rsf\rightarrow\overline{\mathbb{Q}}_{\ell}^{*}$ be a nontrivial additive character. If $\chi_{0}$ is nontrivial, then we have
	\begin{equation*}
		\hyp_{\psi}(\widehat{B},\chi)\cong\gkz_{\chi}(B)\otimes\gauss(\chi_{0},\psi)[d+n]\in D_{c}^{b}(\afs^{n},\overline{\mathbb{Q}}_{\ell}).
	\end{equation*}
	Here $\gauss(\chi_{0},\psi)$ is the one-dimensional $\overline{\mathbb{Q}}_{\ell}$-vector space with a continuous action of $\gal(\overline{\rsf}/\rsf)$ such that the geometric Frobenius acts by the multiplication by the Gauss sum associated to $\chi_{0}$ and $\psi$.
\end{lem}
\begin{proof}
	By the last formula on page~84 of \cite[Proof~of~Theorem~0.8]{fu2016l}, we have the following isomorphism in $D_{c}^{b}(\afs^{n},\overline{\mathbb{Q}}_{\ell})$.
	\begin{equation*}
		\hyp_{\psi}(\widehat{B},\chi)\cong R\pi_{!}(\tau^{*}(\tau_{1}^{*}\kum_{\chi_{1}}\otimes\dots\otimes\tau_{d}^{*}\kum_{\chi_{d}})\otimes g_{B}^{*}j_{\dagger !}\kum_{\chi_{0}^{-1}})\otimes\gauss(\chi_{0},\psi)[d+n].
	\end{equation*}
	By base change theorem for proper pushforward, we have the following isomorphism in $D_{c}^{b}(\afs^{n}_{\trs^{d}},\overline{\mathbb{Q}}_{\ell})$.
	\begin{equation*}
		g_{B}^{*}j_{\dagger !}\kum_{\chi_{0}^{-1}}\cong j_{B!}g_{B}'^{*}\kum_{\chi_{0}^{-1}}.
	\end{equation*}
	By the projection formula, we have the following isomorphism in $D_{c}^{b}(\afs^{n}_{\trs^{d}},\overline{\mathbb{Q}}_{\ell})$.
	\begin{equation*}
		j_{B!}\kum_{\chi}(B)\cong\tau^{*}(\tau_{1}^{*}\kum_{\chi_{1}}\otimes\dots\otimes\tau_{d}^{*}\kum_{\chi_{d}})\otimes j_{B!}g_{B}'^{*}\kum_{\chi_{\chi_{0}^{-1}}}.
	\end{equation*}
	Then, the assertion follows.
\end{proof}

\numberwithin{thm}{subsection}
\numberwithin{con}{subsection}
\numberwithin{cor}{subsection}
\numberwithin{dfn}{subsection}
\numberwithin{exm}{subsection}
\numberwithin{lem}{subsection}
\numberwithin{prp}{subsection}
\numberwithin{rem}{subsection}

\section{Resolutions of singularities}\label{singularity}
We interpret the problem in the language of toric varieties in \Cref{divisor}, and introduce certain types of blow-ups in \Cref{blowup} which we use to construct the resolution of singularities in \Cref{resolution}. As a reference for the toric geometry notions, we refer to \cite{fulton1993introduction}.
\subsection{Divisors corresponding to a matrix}\label{divisor}

We define a divisor corresponding to a matrix which is equal to the complement of the lisse locus of an $\ell$-adic Kummer-type sheaf. We also explain how we detect a resolution of singularities.

We consider $\mathbb{Z}^{d}$ as a lattice in $\mathbb{R}^{d}$. Denote $(\mathbb{Z}^{d})\text{*}$ be the dual lattice of $\mathbb{Z}^{d}$. For a $d$-dimensional nonsingular cone $\sigma$ in $\mathbb{Z}^{d}$, let $\edge(\sigma)$ be the set of edges of $\sigma$. For an edge $\epsilon\in\edge(\sigma)$, let $v_{\epsilon}=(v_{\epsilon i})_{1\leq i\leq d}\in\mathbb{Z}^{d}$ be its unique generator. Let $\{u_{\sigma,\epsilon}\mid\epsilon\in\edge(\sigma)\}$ be the set of elements of $\sigma^{\vee}\cap(\mathbb{Z}^{d})\text{*}$ satisfying
\begin{equation*}
	v_{\epsilon'}^{\top}\cdot u_{\sigma,\epsilon}=\begin{cases}
		1&\epsilon=\epsilon'\\
		0&\epsilon\neq\epsilon'
	\end{cases}.
\end{equation*}
Write $u_{\sigma,\epsilon}=(u_{\sigma,\epsilon i})_{1\leq i\leq d}$ and $V_{\sigma}=[v_{\epsilon i}]_{1\leq i\leq d}$. We note that $(V_{\sigma}^{-1})^{\top}=[u_{\sigma,\epsilon i}]_{1\leq i\leq d}$. Let $\Gamma(\sigma)$ be the finite set $\{t_{\sigma,\epsilon}\mid\epsilon\in\edge(\sigma)\}$. We identify the monoid $\sigma^{\vee}\cap(\mathbb{Z}^{d})\text{*}$ with the commutative semi-group freely generated by the set $\gamma(\sigma)$ via the isomorphism
\begin{equation*}
	\sum_{\epsilon\in\edge(\sigma)}w_{\epsilon}u_{\sigma,\epsilon}\mapsto\prod_{\epsilon\in\edge(\sigma)}t_{\sigma,\epsilon}^{w_{\epsilon}}.
\end{equation*}
Here $w_{\epsilon}\in\mathbb{Z}_{\geq 0}$ for $\epsilon\in\edge(\sigma)$. Then, we denote by $\trv(\sigma)=\spec\Lambda[t_{\sigma,\epsilon}\mid\epsilon\in\edge(\sigma)]$ and note that $\trv(\sigma)$ is the \textit{affine toric variety} corresponding to $\sigma$ over $\Lambda$. The ring homomorphism
\begin{equation*}
	\Lambda[\Gamma(\sigma)]\rightarrow\Lambda[t_{1}^{\pm 1},\dots,t_{d}^{\pm 1}]:t_{\sigma,\epsilon}\mapsto\prod_{i=1}^{d}t_{i}^{u_{\sigma,\epsilon i}}
\end{equation*}
induces an open immersion $\trs^{d}\rightarrow\trv(\sigma)$. Let $B=[b_{ij}]_{1\leq i\leq d;1\leq j\leq n}$ be a matrix whose entries are integers. For $\epsilon\in\edge(\sigma)$, let
\begin{equation*}
	b_{\epsilon i}=\sum_{j=1}^{d}v_{\epsilon j}b_{ij},
\end{equation*}
and let $b_{\epsilon i}^{+}=b_{\epsilon i}-\min\{b_{\epsilon i}\mid 1\leq i\leq n\}$. Let $X_{0},\dots,X_{n}$ be the coordinate of the $n$-dimensional projective space $\prs^{n}=\proj\Lambda[X_{0},\dots,X_{n}]$. We define
\begin{equation*}
	G_{B}(\sigma)=\sum_{i=1}^{n}X_{i}\prod_{\epsilon\in\edge(\sigma)}t_{\sigma,\epsilon}^{b_{\epsilon i}^{+}}\in\Lambda[\Gamma(\sigma)][X_{0},\dots,X_{n}].
\end{equation*}
Let $D_{B}(\sigma)$ be the divisor of $\prs^{n}_{\trv(\sigma)}$ defined by $G_{B}(\sigma)=0$. We note that $D_{B}(\sigma)$ is irreducible since $G_{B}(\sigma)$ is an irreducible polynomial.

\begin{dfn}
	Let $\Sigma$ be a finite set of $d$-dimensional cones in $\mathbb{Z}^{d}$. If the set of the faces of all $\sigma\in\Sigma$ is a fan in $\mathbb{Z}^{d}$, then we say that $\Sigma$ is \textnormal{generable} and denote the fan by $\Delta(\Sigma)$.
\end{dfn}

Let $\Sigma$ be a generable set of cones in $\mathbb{Z}^{d}$. We denote by $\edge(\Sigma)=\cup_{\sigma\in\Sigma}\edge(\sigma)$. For $\sigma,\sigma'\in\Sigma$, we denote by $\Gamma(\sigma,\sigma')=\Gamma(\sigma)\cup\{t_{\sigma,\epsilon}^{-1}\mid\epsilon\in\edge(\sigma)\setminus\edge(\sigma')\}$. Then, the ring isomorphism
\begin{equation}\label{divisor-glueing-morphism}
	\Lambda[\Gamma(\sigma,\sigma')]\rightarrow\Lambda[\Gamma(\sigma',\sigma)]:t_{\sigma,\epsilon}\mapsto\prod_{\epsilon'\in\edge(\sigma')}t_{\sigma',\epsilon'}^{v_{\epsilon'}^{\top}\cdot u_{\sigma,\epsilon}}
\end{equation}
induces $\spec\Lambda[\Gamma(\sigma,\sigma')]\cong\spec\Lambda[\Gamma(\sigma',\sigma)]$. Therefore $\trv(\sigma)$ for all $\sigma\in\Sigma$ glue together to be the \textit{toric variety} corresponding to $\Sigma$ over $\Lambda$ which we denote by $\trv(\Sigma)$.

\begin{lem}\label{divisor-glue-divisor}
	Let $\Sigma$ be a generable set of cones in $\mathbb{Z}^{d}$, and let $B=[b_{ij}]_{1\leq i\leq d;1\leq j\leq n}$ be a matrix whose entries are integers. Then, for any $\sigma,\sigma'\in\Sigma$, the isomorphism \cref{divisor-glueing-morphism} induces an isomorphism $D_{B}(\sigma)\cap\prs^{n}_{\trv(\sigma)}\cong D_{B}(\sigma')\cap\prs^{n}_{\trv(\sigma')}$.
\end{lem}
\begin{proof}
	Note that $b_{\epsilon i}=\sum_{\epsilon'\in\sigma'}(v_{\epsilon}^{\top}\cdot u_{\sigma',\epsilon'})b_{\epsilon' i}$ for all $\epsilon\in\edge(\sigma)$. Let
	\begin{equation*}
		w_{\epsilon}=\sum_{\epsilon'\in\sigma'}(v_{\epsilon}^{\top}\cdot u_{\sigma',\epsilon'})\mid\{b_{\epsilon' i}\mid 1\leq i\leq n\}-\mid\{b_{\epsilon i}\mid 1\leq i\leq n\}.
	\end{equation*}
	We note that $w_{\epsilon}=0$ if $\epsilon\in\edge(\sigma)\cap\edge(\sigma')$. Let $G_{B}(\sigma,\sigma')$ be the inverse image of $G_{B}(\sigma')$ in $\Gamma(\sigma,\sigma')$, then we have
	\begin{equation*}
		G(\sigma,\sigma')/G_{B}(\sigma)=\prod_{\epsilon\in\edge(\sigma)\setminus\edge(\sigma')}t_{\sigma,\epsilon}^{w_{\epsilon}}
	\end{equation*}
	which is invertible in $\Gamma(\sigma,\sigma')$. Then follows the assertions.
\end{proof}

\Cref{divisor-glue-divisor} implies that $D_{B}(\sigma)$ for all $\sigma\in\Sigma$ glue together to be a divisor of $\prs^{n}_{\trv(\Sigma)}$ when $\Sigma$ is generable. We denote this divisor by $D_{B}(\Sigma)$ and note that it is irreducible since it is connected. We denote by $\infty$ the divisor of $\prs^{n}$ defined by $X_{0}$, and denote $\trv(\Sigma)\times\infty$ by $\infty_{\trv(\Sigma)}$. For $\epsilon\in\edge(\Sigma)$, let $D(\Sigma,\epsilon)$ be the closure of the orbit of $\epsilon$ in $\trv(\Sigma)$ by the torus action. We note that $D(\Sigma,\epsilon)\cap\trv(\sigma)$ is the divisor of $\trv(\sigma)$ defined by $t_{\sigma,\epsilon}$ for $\sigma\in\Sigma$. Then, we write
\begin{equation*}
	\overline{D}_{B}(\Sigma)=D_{B}(\Sigma)\cup\infty_{\trv(\Sigma)}\bigcup_{\epsilon\in\edge(\Sigma)}\prs^{n}_{D(\Sigma,\epsilon)}.
\end{equation*}
Set $x_{i}=\frac{X_{i}}{X_{0}}$ for $1\leq i\leq n$, then $\afs^{n}=D_{+}(X_{0})\subset\proj\Lambda[X_{0},\dots,X_{n}]=\prs^{n}$. Let $j:\afs^{n}\rightarrow\prs^{n}$ be the open immersion. Therefore, the open subset $U_{B}\subset\afs^{n}$ defined by \cref{sheaf-kummer-sheaf-lisse-locus} is also an open subset of $\prs^{n}_{\trv(\Sigma)}$. We claim that $U_{B}=\prs^{n}_{\trv(\Sigma)}\setminus\overline{D}_{B}(\Sigma)$ and the explanation follows below.

First, by the fact that $\afs^n=\prs^n\backslash\infty$, we have
	\begin{equation*}
		\afs^n_{\trv(\Sigma)}=\trv(\Sigma)\times\afs^n=\trv(\Sigma)\times\prs^n\backslash\trv(\Sigma)\times\infty=\prs^n_{\trv(\Sigma)}\backslash\infty_{\trv(\Sigma)}.
	\end{equation*}
	At the same time, by the definition of $D(\Sigma,\epsilon)$, it is clear that
	\begin{equation*}
		\trs^d=\left.\trv(\Sigma)\middle\backslash\bigcup_{\epsilon\in\edge(\Sigma)}D(\Sigma,\epsilon)\right..
	\end{equation*}
	Therefore, we have
	\begin{equation*}
		\prs^n_{\trs^d}=\left.\trs^d\times\prs^n=\trv(\Sigma)\times\prs^n\middle\backslash\bigcup_{\epsilon\in\edge(\Sigma)}D(\Sigma,\epsilon)\times\prs^n\right.=\left.\prs^n_{\trv(\Sigma)}\middle\backslash\bigcup_{\epsilon\in\edge(\Sigma)}\prs^n_{D(\Sigma,\epsilon)}\right..
	\end{equation*}
	Here, we note that $\afs^n_{\trs^d}=\afs^n_{\trv(\Sigma)}\cap\prs^n_{\trs^d}$. Next, by the definition of $D_B(\Sigma)$, it is clear that
	\begin{equation*}
		D_B(\Sigma)\cap\afs^n_{\trs^d}=D_B(\Sigma)\cap\trs^d\times\afs^n=\overline{g_B^{-1}(0)}.
	\end{equation*}
	Therefore, we have
	\begin{align*}
		U_B=&\afs^n_{\trs^d}\backslash\overline{g_B^{-1}(0)}=(\afs^n_{\trv(\Sigma)}\cap\prs^n_{\trs^d})\backslash(D_B(\Sigma)\cap\afs^n_{\trs^d})\\
		=&\left.\left(\prs^n_{\trv(\Sigma)}\middle\backslash\infty_{\trv(\Sigma)}\bigcup_{\epsilon\in\edge(\Sigma)}\prs^n_{D(\Sigma,\epsilon)}\right)\right\backslash D_B(\Sigma)=\prs^{n}_{\trv(\Sigma)}\setminus\overline{D}_{B}(\Sigma).
	\end{align*}

\begin{dfn}
	Let $B=[b_{ij}]_{1\leq i\leq d;1\leq j\leq n}$ be a matrix whose entries are integers.
	\begin{enumerate}
		\item For a $d$-dimensional nonsingular cone $\sigma$ in $\mathbb{Z}^{d}$, if there exists a permutation $\perm$ of degree $n$, such that
		\begin{equation*}
			b_{\epsilon\perm(1)}\leq\dots\leq b_{\epsilon\perm(n)}
		\end{equation*}
		for all $\epsilon\in\edge(\sigma)$, then we say that $B$ is \textnormal{$\sigma$-good}.
		
		\item For a generable set $\Sigma$ of cones in $\mathbb{Z}^{d}$, if $B$ is $\sigma$-good for every $\sigma\in\Sigma$, then we say that $B$ is \textnormal{$\Sigma$-good}.
	\end{enumerate}
\end{dfn}

\begin{prp}\label{divisor-good-is-normal}
	Let $\Sigma$ be a generable set of cones in $\mathbb{Z}^{d}$, and let $B=[b_{ij}]_{1\leq i\leq d;1\leq j\leq n}$ be a matrix whose entries are integers. If $B$ is $\Sigma$-good, then $\overline{D}_{B}(\Sigma)$ is a divisor with simple normal crossings.
\end{prp}
\begin{proof}
	For any $\sigma\in\Sigma$, we may assume that $b_{\epsilon 1}\leq\dots\leq b_{\epsilon n}$, in which case we have
	\begin{equation*}
		G_{B}(\sigma)=\biggl(\prod_{\epsilon\in\edge(\sigma)}t_{\sigma,\epsilon}\biggr)^{b_{\epsilon 1}}\biggl(X_{1}+\sum_{i=2}^{n}X_{i}\prod_{\epsilon\in\edge(\sigma)}t_{\sigma,\epsilon}^{b_{\epsilon i}-b_{\epsilon 1}}\biggr).
	\end{equation*}
	The image of $\overline{D}_{B}(\Sigma)\cap\prs^{n}_{\trv(\sigma)}$ by the automorphism of $\prs^{n}_{\trv(\sigma)}$ defined by the $\Lambda[\Gamma(\sigma)]$-automorphism
	\begin{equation*}
	\begin{aligned}
		\Lambda[\Gamma(\sigma)][X_{i}\mid 0\leq i\leq n]&\rightarrow\Lambda[\Gamma(\sigma)][X_{i}\mid 0\leq i\leq n]\\
		X_{i}&\mapsto\begin{cases}
			X_{1}+\sum_{i=2}^{n}X_{i}\prod_{\epsilon\in\edge(\sigma)}t_{\sigma,\epsilon}^{b_{\epsilon i}-b_{\epsilon 1}}&i=1\\
			X_{i}&2\leq i\leq n
		\end{cases}
	\end{aligned}
	\end{equation*}
	is the divisor of $\prs^{n}_{\trv(\sigma)}$ defined by $X_{0}X_{1}\prod_{\epsilon\in\edge(\sigma)}t_{\sigma,\epsilon}$ which is a divisor with simple normal crossings. Then, the assertion follows.
\end{proof}

\subsection{Standard blow-ups of a generable set}\label{blowup}

We introduce certain types of blow-ups along the singularities of the complement of the lisse locus of an $\ell$-adic Kummer-type sheaf and study their properties
which we then use to construct the resolution of singularities.

Let $\Sigma$ be a generable set of cones in $\mathbb{Z}^{d}$. For $\epsilon\in\edge(\Sigma)$, we denote by
\begin{equation*}
	\Sigma(\epsilon)=\{\sigma\in\Sigma\mid\epsilon\in\edge(\sigma)\}\subset\Sigma.
\end{equation*}

\begin{dfn}
	Let $\Sigma$ be a generable set of cones in $\mathbb{Z}^{d}$. Let $\epsilon_{1}\neq\epsilon_{2}\in\edge(\Sigma)$ satisfying $\Sigma(\epsilon_{1})\cap\Sigma(\epsilon_{2})\neq\varnothing$.
	\begin{enumerate}
		\item We define the \textnormal{standard blow-up} of $\Sigma$ along $\epsilon_{1},\epsilon_{2}$ to be the set
		\begin{equation*}
			\widetilde{\Sigma}=(\Sigma\setminus(\Sigma(\epsilon_{1})\cup\Sigma(\epsilon_{2})))\bigcup_{\sigma\in\Sigma(\epsilon_{1})\cap\Sigma(\epsilon_{2})}\{\sigma(\epsilon_{1}),\sigma(\epsilon_{2})\}
		\end{equation*}
		and write $\widetilde{\Sigma}\rightarrow\Sigma$. Here $\sigma(\epsilon_{i})$ is the cone in $\mathbb{Z}^{d}$ generated by the set of vectors
		\begin{equation*}
			\{v_{\epsilon_{1}}+v_{\epsilon_{2}}\}\cup\{v_{\epsilon}\mid\epsilon\in\edge(\sigma)\setminus\{\epsilon_{j}\};\{j\}=\{1,2\}\setminus\{i\}\}.
		\end{equation*}
		
		\item We define the \textnormal{exceptional edge} of $\widetilde{\Sigma}\rightarrow\Sigma$ to be the cone in $\mathbb{Z}^{d}$ generated by $v_{\epsilon_{1}}+v_{\epsilon_{2}}$.
	\end{enumerate}
\end{dfn}

\begin{prp}
	Let $\Sigma$ be a generable set of cones in $\mathbb{Z}^{d}$, and let $\epsilon_{1}\neq\epsilon_{2}\in\edge(\Sigma)$ satisfying $\Sigma(\epsilon_{1})\cap\Sigma(\epsilon_{2})\neq\varnothing$. Then, the following statements are true.
	\begin{enumerate}
		\item The standard blow-up $\widetilde{\Sigma}$ of $\Sigma$ along $\epsilon_{1},\epsilon_{2}$ is generable.
		\item Let $\epsilon_{\ex}$ be the exceptional edge of $\widetilde{\Sigma}\rightarrow\Sigma$, then $\edge(\widetilde{\Sigma})=\edge(\Sigma)\cup\{\epsilon_{\ex}\}$.
		\item $\trv(\widetilde{\Sigma})$ is the blow-up of $\trv(\Sigma)$ along $D(\Sigma,\epsilon_{1})\cap D(\Sigma,\epsilon_{2})$.
		\item $D(\widetilde{\Sigma},\epsilon)$ is the strict transform of $D(\Sigma,\epsilon)$ corresponding to the blow-up $\trv(\widetilde{\Sigma})\rightarrow\trv(\Sigma)$ for any $\epsilon\in\edge(\Sigma)$.
		\item $D(\widetilde{\Sigma},\epsilon_{\ex})$ is the exceptional divisor of the blow-up $\trv(\widetilde{\Sigma})\rightarrow\trv(\Sigma)$.
	\end{enumerate}
\end{prp}
\begin{proof}
	We first prove 1. For any $\sigma\in\Sigma(\epsilon_{1})\cap\Sigma(\epsilon_{2})$, we note that
	\begin{equation*}
		\sigma(\epsilon_{1})\cap\sigma(\epsilon_{2})=\sigma(\epsilon_{1})\cap(u_{\sigma,\epsilon_{1}}+u_{\sigma,\epsilon_{2}})^{\perp}=\sigma(\epsilon_{2})\cap(u_{\sigma,\epsilon_{1}}+u_{\sigma,\epsilon_{2}}).
	\end{equation*}
	For $\sigma'\in(\Sigma(\epsilon_{1})\cap\Sigma(\epsilon_{2}))\setminus\{\sigma\}$ and $i\in\{1,2\}$, we note that
	\begin{equation*}
	\begin{aligned}
		\sigma(\epsilon_{i})\cap\sigma'(\epsilon_{i})&=\sigma(\epsilon_{i})\cap(\sigma\cap\sigma')=\sigma'(\epsilon_{i})\cap(\sigma\cap\sigma'),\\
		\sigma(\epsilon_{1})\cap\sigma'(\epsilon_{2})&=(\sigma(\epsilon_{1})\cap\sigma(\epsilon_{2}))\cap(\sigma(\epsilon_{2})\cap\sigma'(\epsilon_{2})).
	\end{aligned}
	\end{equation*}
	Furthermore, for $\sigma''\in\Sigma\setminus(\Sigma(\epsilon_{1})\cup\Sigma(\epsilon_{2}))$, we note that
	\begin{equation*}
		\sigma(\epsilon_{i})\cap\sigma''=\sigma(\epsilon_{i})\cap(\sigma\cap\sigma'')=(\sigma(\epsilon_{i})\cap\sigma)\cap\sigma''.
	\end{equation*}
	Therefore $\widetilde{\Sigma}$ is generable.
	
	Next, we prove 2. For $\sigma\in\Sigma(\epsilon_{1})\cap\Sigma(\epsilon_{2})$ and $i\neq j\in\{1,2\}$, we note that $\edge(\sigma(\epsilon_{i}))=(\edge(\sigma)\cup\{\epsilon_{\ex}\})\setminus\{\epsilon_{j}\}$. Then the assertion follows.
	
	We move on to prove 3. For $\sigma\in\Sigma(\epsilon_{1})\cap\Sigma(\epsilon_{2})$ and $i\in\{1,2\}$, let $\mathrm{bl}_{\sigma(\epsilon_{i})}:\trv(\sigma(\epsilon_{i}))\rightarrow\trv(\sigma)$ be the morphism defined by the ring homomorphism
	\begin{equation*}
		\Lambda[\Gamma(\sigma)]\rightarrow\Lambda[\Gamma(\sigma(\epsilon_{i}))]:t_{\sigma,\epsilon}\mapsto\begin{cases}
				t_{\sigma(\epsilon_{i}),\epsilon_{i}}\cdot t_{\sigma(\epsilon_{i}),\epsilon_{\ex}}&\epsilon=\epsilon_{i}\\
				t_{\sigma,\epsilon}&\epsilon\in\edge(\sigma)\setminus\{\epsilon_{i}\}
			\end{cases}.
	\end{equation*}
	For $\sigma\in\Sigma\setminus(\Sigma(\epsilon_{1})\cup\Sigma(\epsilon_{2}))$, let $\mathrm{bl}_{\sigma}:\trv(\sigma)\rightarrow\trv(\sigma)$ be the identity. Then $\mathrm{bl}_{\sigma}$ for all $\sigma\in\edge(\widetilde{\Sigma})$ glue together to be a morphism $\mathrm{bl}:\trv(\widetilde{\Sigma})\rightarrow\trv(\Sigma)$, which gives the morphism of blow-up.
	
	To continue, we prove 4. For $\sigma\in\Sigma(\epsilon_{1})\cap\Sigma(\epsilon_{2})$ and $i\in\{1,2\}$, the inverse image of $D(\Sigma,\epsilon)\cap\trv(\sigma)$ along $\mathrm{bl}_{\sigma(\epsilon_{i})}$ is the divisor of $\trv(\sigma(\epsilon_{i}))$ defined by $t_{\sigma(\epsilon_{i}),\epsilon_{i}}\cdot t_{\sigma(\epsilon_{i}),\epsilon_{\ex}}$. Then the assertion follows.
	
	Finally, we prove 5. For $\sigma\in\Sigma(\epsilon_{1})\cap\Sigma(\epsilon_{2})$, we note that $D(\widetilde{\Sigma},\epsilon_{\ex})\cap\trv(\sigma)$ is the divisor of $\trv(\sigma(\epsilon_{i}))$ defined by $t_{\sigma(\epsilon_{i}),\epsilon_{\ex}}$. Then the assertion follows.
\end{proof}

\begin{prp}
	Let $\widetilde{\Sigma}\rightarrow\Sigma$ be the standard blow-up of generable set of cones in $\mathbb{Z}^{d}$ along $\epsilon_{1},\epsilon_{2}\in\edge(\Sigma)$, and let $B=[b_{ij}]_{1\leq i\leq d;1\leq j\leq n}$ be a matrix whose entries are integers. Then, the following statements are true.
	\begin{enumerate}
		\item $D_{B}(\widetilde{\Sigma})$ is the strict transform of $D_{B}(\Sigma)$ corresponding to $\prs^{n}_{\trv(\widetilde{\Sigma})}\rightarrow\prs^{n}_{\trv(\Sigma)}$.
		\item $\overline{D}_{B}(\widetilde{\Sigma})$ is the total transform of $\overline{D}_{B}(\Sigma)$ corresponding to $\prs^{n}_{\trv(\widetilde{\Sigma})}\rightarrow\prs^{n}_{\trv(\Sigma)}$.
		\item For $\sigma\in\Sigma(\epsilon_{1})\cap\Sigma(\epsilon_{2})$, if $B$ is $\sigma$-good, then $B$ is $\sigma(\epsilon_{1})$-good as well as $\sigma(\epsilon_{2})$-good.
	\end{enumerate}
\end{prp}
\begin{proof}
	The first assertion follows from the construction of the blow-up $\trv(\widetilde{\Sigma})\rightarrow\trv(\Sigma)$.
	
	Next, we note that $\infty_{\trv(\widetilde{\Sigma})}$ is the total transform corresponding to $\prs^{n}_{\trv(\widetilde{\Sigma})}\rightarrow\prs^{n}_{\trv(\Sigma)}$. Then, the second assertion follows from 3.2.2.
	
	Finally, for $1\leq i\leq n$, we note that $b_{\epsilon_{\ex} i}=b_{\epsilon_{1}i}+b_{\epsilon_{2}i}$. Then, the third assertion follows from 3.2.2.
\end{proof}

\subsection{Resolutions of singularities}\label{resolution}

We prove the existence of the singularities of the complement of the lisse locus of an $\ell$-adic Kummer-type sheaf. We first deal with the $n=2$ case and then extend the result to the general case by induction.

Let $\Sigma$ be a generable subset of cones in $\mathbb{Z}^{d}$, and let $B=[b_{ij}]_{1\leq i\leq d;1\leq j\leq n}$ be a matrix whose entries are integers. We denote by $\Sigma_{\bad}(B)$ the set
\begin{equation*}
	\Sigma_{\bad}(B)=\{\sigma\in\Sigma\mid\text{$B$ is not $\sigma$-good}\}
\end{equation*}
and note that it is generable since it is a subset of $\Sigma$. When $n=2$, we define
\begin{equation*}
\begin{aligned}
	\mu_{B}(\Sigma)&=\begin{cases}
		0&\Sigma_{\bad}(B)=\varnothing\\
		\max\{|b_{\epsilon 1}-b_{\epsilon 2}|\mid\epsilon\in\edge(\Sigma_{\bad}(B))\}&\Sigma_{\bad}(B)\neq\varnothing
	\end{cases},\\
	\edge_{B}(\Sigma)&=\begin{cases}
		\varnothing&\Sigma_{\bad}(B)=\varnothing\\
		\{\epsilon\in\edge(\Sigma_{\bad}(B))\mid|b_{\epsilon 1}-b_{\epsilon 2}|=\mu_{B}(\Sigma)\}&\Sigma_{\bad}(B)\neq\varnothing
	\end{cases}.
\end{aligned}
\end{equation*}
Intuitively $\mu_{B}(\Sigma)$ is the maximal multiplicity of singularities of $\overline{D}_{B}(\Sigma)$ and the worst singularities are located at $\edge_{B}(\Sigma)$. For $\epsilon\in\edge_{B}(\Sigma)$, if $\Sigma_{\bad}(\Sigma)\neq\varnothing$, then we denote by
\begin{equation*}
	\edge_{B}(\Sigma,\epsilon)=\{\epsilon'\in\edge(\Sigma_{\bad}(B))\cap\edge(\Sigma(\epsilon))\mid(b_{\epsilon 1}-b_{\epsilon 2})(b_{\epsilon' 1}-b_{\epsilon' 2})<0\}
\end{equation*}
which is a non-empty subset of $\edge(\Sigma_{\bad}(B))$. Then, we define
\begin{equation*}
	\nu_{B}(\Sigma)=\begin{cases}
		0&\Sigma_{\bad}(B)=\varnothing\\
		\sum_{\epsilon\in\edge_{B}(\Sigma)}\#\edge_{B}(\Sigma,\epsilon)&\Sigma_{\bad}(B)\neq\varnothing
	\end{cases}.
\end{equation*}
To observe $\edge_{B}(\Sigma,\epsilon_{0})$ more closely for some $\epsilon_{0}\in\edge_{B}(\Sigma)$, we look at \Cref{resolution-figure-edges}. The gray area stands for the subset of $\edge(\Sigma_{\bad}(B))$ consisting of $\epsilon\in\edge(\Sigma_{\bad}(B))$ satisfying
\begin{equation*}
	(b_{\epsilon 1}-b_{\epsilon 2})(b_{\epsilon_{0}1}-b_{\epsilon_{0}2})\geq 0,
\end{equation*}
and the black area stands for $\edge_{B}(\Sigma,\epsilon_{0})$.

\begin{figure}[H]
\centering
\begin{tikzpicture}
	\filldraw[lightgray](1.5,0.75)rectangle(7.75,2.75);
	\filldraw[lightgray](7.75,1.75)circle(1);
	\filldraw[white](2,1)rectangle(7.75,2);
	\filldraw(6,1)rectangle(7.75,2);
	\filldraw(6,3.5)arc[start angle=90,end angle=270,x radius=1.5,y radius=1.5];
	\filldraw(7.75,1)arc[start angle=-90,end angle=90,x radius=0.5,y radius=0.5];
	\filldraw[white](0,0)rectangle(8,0.75);
	\filldraw[white](0,2.75)rectangle(8,4);
	\filldraw[lightgray](1.5,0.75)rectangle(7.75,1);
	\filldraw[lightgray](1.5,2)rectangle(7.75,2.75);
	\filldraw(6.25,2.5)circle(1pt);
	\draw(6,0.5)--(8,0.5);
	\draw(8,3.5)--(6,3.5);
	\draw(7.75,2)--(2,2)--(2,1)--(7.75,1);
	\draw(7.75,2.75)--(1.5,2.75)--(1.5,0.75)--(7.75,0.75);
	\draw(8,3.75)--(8,4)--(0,4)--(0,0)--(8,0)--(8,0.25);
	\draw(8,0.25)arc[start angle=-90,end angle=90,x radius=1.75,y radius=1.75];
	\draw(6,3.5)arc[start angle=90,end angle=270,x radius=1.5,y radius=1.5];
	\draw(7.75,0.75)arc[start angle=-90,end angle=90,x radius=1,y radius=1];
	\draw(7.75,1)arc[start angle=-90,end angle=90,x radius=0.5,y radius=0.5];
	\draw(8,0.5)arc[start angle=-90,end angle=90,x radius=1.5,y radius=1.5];
	\draw(0,4)node[below right]{$\edge(\Sigma)$};
	\draw(1.5,2.75)node[below right]{$\edge(\Sigma_{\bad}(B))$};
	\draw(6.25,2.5)node[right]{$\epsilon_{0}$};
	\draw(8,3.5)node[below left]{$\edge(\Sigma(\epsilon_{0}))$};
	\draw(7.75,2)node[below left,white]{$\edge_{B}(\Sigma,\epsilon_{0})$};
\end{tikzpicture}
\caption{}
\label{resolution-figure-edges}
\end{figure}

\begin{dfn}
	Let $\Sigma$ be a generable set of cones in $\mathbb{Z}^{d}$, and let $B=[b_{ij}]_{1\leq i\leq d;1\leq j\leq 2}$ be a matrix whose entries are integers. We say that a standard blow-up $\widetilde{\Sigma}\rightarrow\Sigma$ along $\epsilon_{1},\epsilon_{2}\in\edge(\Sigma)$ is \textnormal{$B$-good} if the following conditions are satisfied.
	\begin{enumerate}
		\item $\epsilon_{1}\in\edge_{B}(\Sigma)$ as well as $\epsilon_{2}\in\edge_{B}(\Sigma,\epsilon_{1})$.
		\item $|b_{\epsilon_{2}1}-b_{\epsilon_{2}2}|=\max\{|b_{\epsilon 1}-b_{\epsilon 2}|\mid\epsilon\in\edge_{B}(\Sigma,\epsilon_{1})\}$.
	\end{enumerate}
\end{dfn}

\begin{prp}
	Let $\Sigma$ be a generable set of cones in $\mathbb{Z}^{d}$, and let $B=[b_{ij}]_{1\leq i\leq d;1\leq j\leq 2}$ be a matrix whose entries are integers. If $\widetilde{\Sigma}\rightarrow\Sigma$ is a $B$-good blow-up, then $\mu_{B}(\widetilde{\Sigma})\leq\mu_{B}(\Sigma)$. Furthermore, if $\mu_{B}(\widetilde{\Sigma})=\mu_{B}(\Sigma)$, then $\nu_{B}(\widetilde{\Sigma})<\nu_{B}(\Sigma)$
\end{prp}
\begin{proof}
	Assume that $\widetilde{\Sigma}\rightarrow\Sigma$ is the $B$-good blow-up along $\epsilon_{1}\in\edge_{B}(\Sigma)$ and $\epsilon_{2}\in\edge_{B}(\Sigma,\epsilon_{1})$, then
	\begin{equation*}
		|b_{\epsilon_{2}1}-b_{\epsilon_{2}2}|=\max\{|b_{\epsilon 1}-b_{\epsilon 2}|\mid\epsilon\in\edge_{B}(\Sigma,\epsilon_{1})\}.
	\end{equation*}
	Let $\epsilon_{\ex}$ be the exceptional edge of $\widetilde{\Sigma}\rightarrow\Sigma$, then
	\begin{equation*}
		|b_{\epsilon_{\ex}1}-b_{\epsilon_{\ex}2}|=||b_{\epsilon_{1}1}-b_{\epsilon_{1}2}|-|b_{\epsilon_{2}1}-b_{\epsilon_{2}}2||<|b_{\epsilon_{1}1}-b_{\epsilon_{1}2}|,
	\end{equation*}
	which implies that $\epsilon_{\ex}\notin\edge_{B}(\widetilde{\Sigma})$. On the other hand, we note that
	\begin{equation*}
		\widetilde{\Sigma}_{\bad}(B)\subset\{\sigma(\epsilon_{1}),\sigma(\epsilon_{2})\mid\sigma\in\Sigma_{\bad}(B)\cap\Sigma(\epsilon_{1})\cap\Sigma(\epsilon_{2})\}\cup(\Sigma_{\bad}(B)\setminus(\Sigma(\epsilon_{1})\cup\Sigma(\epsilon_{2}))),
	\end{equation*}
	which implies that $\edge(\widetilde{\Sigma}_{\bad}(B))\subset\edge(\Sigma_{\bad}(B))$, and hence $\mu_{B}(\widetilde{\Sigma})\leq\mu_{B}(\Sigma)$. Next, we assume that $\mu_{B}(\widetilde{\Sigma})=\mu_{B}(\Sigma)$ and prove that $\nu_{B}(\widetilde{\Sigma})<\nu_{B}(\Sigma)$. As a consequence of the discussion above, we note that $\edge_{B}(\widetilde{\Sigma})\subset\edge_{B}(\Sigma)$. For $\epsilon\in\edge_{B}(\widetilde{\Sigma})\setminus\edge_{B}(\Sigma)$, we note that $\edge_{B}(\widetilde{\Sigma},\epsilon)=\edge_{B}(\Sigma,\epsilon)$.
	
	If $\epsilon_{2}\in\edge_{B}(\Sigma)$, then $b_{\epsilon_{\ex}1}=b_{\epsilon_{1}1}+b_{\epsilon_{2}1}=b_{\epsilon_{1}2}+b_{\epsilon_{2}2}=b_{\epsilon_{\ex}2}$, which implies that $\epsilon_{\ex}\notin\edge_{B}(\widetilde{\Sigma},\epsilon)$ for any $\epsilon\in\edge_{B}(\widetilde{\Sigma})$, and hence $\edge_{B}(\widetilde{\Sigma},\epsilon)\subset\edge_{B}(\Sigma,\epsilon)$. Particularly, we note that
	\begin{equation*}
		\edge_{B}(\widetilde{\Sigma},\epsilon_{1})\subset\edge_{B}(\Sigma,\epsilon_{1})\setminus\{\epsilon_{2}\},\ \edge_{B}(\widetilde{\Sigma},\epsilon_{2})\subset\edge_{B}(\Sigma,\epsilon_{2})\setminus\{\epsilon_{1}\},
	\end{equation*}
	which implies that $\nu_{B}(\widetilde{\Sigma})\leq\nu_{B}(\Sigma)-2<\nu_{B}(\Sigma)$.
	
	If $\epsilon_{2}\notin\edge_{B}(\Sigma)$, then $\edge_{B}(\Sigma,\epsilon_{1})\cap\edge_{B}(\Sigma)=\varnothing$. Therefore, for any $\epsilon\in\edge_{B}(\widetilde{\Sigma})\cap\edge(\Sigma(\epsilon_{1}))$, we have $(b_{\epsilon 1}-b_{\epsilon 2})(b_{\epsilon_{1}1}-b_{\epsilon_{1}2})>0$. On the other hand, the fact that $|b_{\epsilon_{1}1}-b_{\epsilon_{1}2}|>|b_{\epsilon_{2}1}-b_{\epsilon_{2}2}|$ implies that $(b_{\epsilon_{1}1}-b_{\epsilon_{1}2})(b_{\epsilon_{\ex}1}-b_{\epsilon_{\ex}2})>0$, and hence $(b_{\epsilon 1}-b_{\epsilon 2})(b_{\epsilon_{\ex}1}-b_{\epsilon_{\ex}2})>0$. This shows that $\epsilon_{\ex}\notin\edge_{B}(\widetilde{\Sigma},\epsilon)$, and hence $\edge_{B}(\widetilde{\Sigma},\epsilon)\subset\edge_{B}(\Sigma,\epsilon)$. Particularly, we note that
	\begin{equation*}
		\edge_{B}(\widetilde{\Sigma},\epsilon_{1})\subset\edge_{B}(\Sigma,\epsilon_{1})\setminus\{\epsilon_{2}\},
	\end{equation*}
	which implies that $\nu_{B}(\widetilde{\Sigma})\leq\nu_{B}(\Sigma)-1<\nu_{B}(\Sigma)$.
\end{proof}

For two generable sets $\Sigma,\widetilde{\Sigma}$ of cones in $\mathbb{Z}^{d}$, if there exist $\widetilde{\Sigma}^{(0)},\dots,\widetilde{\Sigma}^{(r)}$ where $\Sigma=\widetilde\Sigma^{(0)}$ and $\widetilde{\Sigma}=\widetilde{\Sigma}^{(r)}$, such that $\widetilde{\Sigma}^{(i)}\rightarrow\widetilde{\Sigma}^{(i-1)}$ is a standard blow-up for every $i\in\{1,\dots,r\}$, then we write
\begin{equation*}
	\widetilde{\Sigma}=\widetilde{\Sigma}^{(r)}\rightarrow\dots\rightarrow\widetilde{\Sigma}^{(0)}=\Sigma
\end{equation*}
or simply $\widetilde{\Sigma}\rightsquigarrow\Sigma$, and call it a \textit{series} of standard blow-up. If moreover $\widetilde{\Sigma}^{(i)}\rightarrow\widetilde{\Sigma}^{(i-1)}$ is $B$-good for every $i\in\{1,\dots,r\}$, then we say that $\widetilde{\Sigma}\rightsquigarrow\Sigma$ is a \textit{series} of $B$-good blow-ups.

\begin{cor}
	Let $\Sigma$ be a generable set of cones in $\mathbb{Z}^{d}$, and let $B=[b_{ij}]_{1\leq i\leq d;q\leq j\leq 2}$ be a matrix whose entries are integers. If $B$ is not $\Sigma$-good, then the following statements are true.
	\begin{enumerate}
		\item There exists a series of $B$-good blow-ups $\widetilde{\Sigma}\rightsquigarrow\Sigma$, such that $\mu_{B}(\widetilde{\Sigma})<\mu_{B}(\Sigma)$.
		\item There exists a series of $B$-good blow-ups $\widetilde{\Sigma}\rightsquigarrow\Sigma$, such that $B$ is $\widetilde{\Sigma}$-good.
	\end{enumerate}
\end{cor}
\begin{proof}
	To prove the first assertion, we argue by induction on $\nu_{B}(\Sigma)$. First, we assume that $\nu_{B}(\Sigma)=1$. Let $\widetilde{\Sigma}\rightarrow\Sigma$ be a $B$-good blow-up. In fact, such blow-up exists. To explain this, we first fix an arbitrary $\epsilon_1\in\edge_{B}(\Sigma)$. This can always be done, namely, $\edge_{B}(\Sigma)\neq\varnothing$, because by definition of $\edge_{B}(\Sigma)$, it is empty if and only if $\Sigma_{\bad}(B)=\varnothing$, in which case we have no singularities and we are done. In other words, as long as there are singularities to resolve, we have $\edge_{B}(\Sigma)\neq\varnothing$. Next, with $\epsilon_1\in\edge_{B}(\Sigma)$ fixed, we choose $\epsilon_2$ from $\edge_{B}(\Sigma,\epsilon_{1})$ to be the one satisfying
		\begin{equation*}
			|b_{\epsilon_{2}1}-b_{\epsilon_{2}2}|=\max\{|b_{\epsilon 1}-b_{\epsilon 2}|\mid\epsilon\in\edge_{B}(\Sigma,\epsilon_{1})\}.
		\end{equation*}
		This also can always be done, since $\edge_{B}(\Sigma,\epsilon_{1})\neq\varnothing$ and $\#\edge_{B}(\Sigma,\epsilon_{1})<\infty$. To be more precise, we know that $\#\edge_{B}(\Sigma,\epsilon_{1})=1$ and $\epsilon_2$ is the only element of $\edge_{B}(\Sigma,\epsilon_{1})$. The first claim is ensured by the definition of $\edge_{B}(\Sigma,\epsilon_{1})$. In other words, $\edge_{B}(\Sigma,\epsilon_{1})\neq\varnothing$ as long as $\epsilon_{1}\in\edge_B(\Sigma)$. The second claim is ensured by the finiteness of the number of edges of $\Sigma$, and $\edge_{B}(\Sigma,\epsilon_{1})$ is a subset of the set of all edges of $\Sigma$.
		
		If $\mu_{B}(\widetilde{\Sigma})=\mu_{B}(\Sigma)$, then $\nu_{B}(\widetilde{\Sigma})<\nu_{B}(\Sigma)$, which implies that that $\nu_{B}(\widetilde{\Sigma})=0$ and leads to contradiction. Therefore, we have $\mu_{B}(\widetilde{\Sigma})<\mu_{B}(\Sigma)$. Next, assume that the argument is true when $1\leq\nu_{B}(\Sigma)\leq\nu$, we prove that it is also true when $\nu_{B}(\Sigma)=\nu+1$. Let $\widetilde{\Sigma}^{(1)}\rightarrow\Sigma$ be a $B$-good blow-up. If $\mu_{B}(\widetilde{\Sigma}^{(1)})=\mu_{B}(\Sigma)$, then $\nu_{B}(\widetilde{\Sigma}^{(1)})<\nu_{B}(\Sigma)$. By the induction hypothesis, there exists a series of $B$-good blow-ups $\widetilde{\Sigma}\rightsquigarrow\widetilde{\Sigma}^{(1)}$, such that $\mu_{B}(\widetilde{\Sigma})<\mu_{B}(\widetilde{\Sigma}^{(1)})=\mu_{B}(\Sigma)$, which means that $\widetilde{\Sigma}\rightsquigarrow\Sigma$ is the required series of $B$-good blow-ups.
	
	To prove the second assertion, we argue by induction on $\mu_{B}(\Sigma)$. First, assume that $\mu_{B}(\Sigma)=1$. Let $\widetilde{\Sigma}\rightsquigarrow\Sigma$ be the series of $B$-good blow-ups such that $\mu_{B}(\widetilde{\Sigma})<\mu_{B}(\Sigma)$, which means that $\mu_{B}(\widetilde{\Sigma})=0$, and hence $B$ is $\widetilde{\Sigma}$-good. Next, assume that the argument is true for $1\leq\mu_{B}(\Sigma)\leq\mu$, we prove that it is also true when $\mu_{B}(\Sigma)=\mu+1$. Let $\widetilde{\Sigma}^{(1)}\rightsquigarrow\Sigma$ be the series of $B$-good blow-ups such that $\mu_{B}(\widetilde{\Sigma}^{(1)})<\mu_{B}(\Sigma)$. By the induction hypothesis, there exists a series of $B$-good blow-ups $\widetilde{\Sigma}\rightsquigarrow\widetilde{\Sigma}^{(1)}$, such that $B$ is $\widetilde{\Sigma}$-good, which means that $\widetilde{\Sigma}\rightsquigarrow\Sigma$ is the required series of $B$-good blow-ups.
\end{proof}

Next, we consider the case where $n\in\mathbb{Z}_{\geq 2}$.

\begin{thm}\label{resolution-existence}
	Let $\Sigma$ be a generable set of cones in $\mathbb{Z}^{d}$, and let $B=[b_{ij}]_{1\leq i\leq d;1\leq j\leq n}$. If $B$ is not $\Sigma$-good, then there exists a series of standard blow-ups $\widetilde{\Sigma}\rightsquigarrow\Sigma$, such that $B$ is $\widetilde{\Sigma}$-good.
\end{thm}
\begin{proof}
	For any $1\leq i<j\leq n$, we denote by $B[i;j]$ the matrix consisting of the columns of $B$ whose indices are $i,j$. Then, we denote by
	\begin{equation*}
		\badpair_{B}(\Sigma)=\{(i,j)\mid 1\leq i<j\leq n;\Sigma_{\bad}(B[i;j])\neq\varnothing\}.
	\end{equation*}
	For any standard blow-up $\widetilde{\Sigma}\rightarrow\Sigma$, we note that $\badpair_{B}(\widetilde{\Sigma})\subset\badpair_{B}(\Sigma)$. We argue by induction on $\#\badpair_{B}(\Sigma)$. First, assume that $\badpair=\{(i,j)\}$. Let $\widetilde{\Sigma}\rightsquigarrow\Sigma$ be the series of $B[i;j]$-good blow-ups such that $B[i;j]$ is $\widetilde{\Sigma}$-good. Then, we have
	\begin{equation*}
		\badpair_{B}(\widetilde{\Sigma})\subset\badpair_{B}(\Sigma)\setminus\{i,j\},
	\end{equation*}
	and hence $\badpair_{B}(\widetilde{\Sigma})=\varnothing$. Next, assume that the argument is true when $1\leq\#\badpair_{B}(\Sigma)\leq b$, we prove that it is also true when $\#\badpair_{B}(\Sigma)=b+1$. For $(i,j)\in\badpair_{B}(\Sigma)$, there exists a series of $B$-good blow-ups $\widetilde{\Sigma}^{(1)}\rightsquigarrow\Sigma$, such that $\badpair_{B}(\widetilde{\Sigma}^{(1)})\subset\badpair_{B}(\Sigma)\setminus\{i,j\}$. By the induction hypothesis, there exists a series of standard blow-ups $\widetilde{\Sigma}\rightsquigarrow\widetilde{\Sigma}^{(1)}$, such that $B$ is $\widetilde{\Sigma}$-good, which implies that $\widetilde{\Sigma}\rightsquigarrow\Sigma$ is the required series of standard blow-ups.
\end{proof}

\section{Characteristic cycles of \texorpdfstring{$\ell$}{l}-adic sheaves}\label{characteristic}
We prove that the characteristic cycle of an $\ell$-adic GKZ-type sheaf is equal to the direct image of the characteristic cycle of the corresponding $\ell$-adic Kummer-type sheaf. We calculate the dimension of certain cycles in \Cref{dimension} and prove that the union of the cycles are the direct image of the singular support of the $\ell$-adic Kummer sheaf in Subsection 4.2, which we use to verify that the requirements for applying the formula \cref{introduction-direct-image-formula-cycle} calculating the characteristic cycle of the direct image of an $\ell$-adic sheaf are met. We apply the formula to our case in \Cref{square}.
\subsection{Dimension of closed conical subsets}\label{dimension}

We define some closed conical subsets of the cotangent bundle of the projective space and calculate their dimensions. Here, a closed subset is called \textit{conical} if it is stable under the action of the multiplicative group $\mathbb G_m$.

For $0\leq i\leq n$, we denote by $\afs^{n}_{i}=D_{+}(X_{i})$ where $D_{+}(X_{i})=\spec\Lambda\bigl[\frac{X_{j}}{X_{i}}\bigm|j\in\{0,\dots,n\}\setminus\{i\}\bigr]$ is an open subset of $\prs^{n}=\proj\Lambda[X_{0},\dots,X_{n}]$. Note that $\afs^{n}_{i}\times_{\prs^{n}}T^{*}\prs^{n}\cong T^{*}\afs^{n}_{i}$. Denote by
\begin{equation*}
	\coor(i)=\bigl\{\tfrac{X_{j}}{X_{i}},\xi\bigl(\tfrac{X_{j}}{X_{i}}\bigr)\bigm|j\in\{0,\dots,n\}\setminus\{i\}\bigr\}
\end{equation*}
which is a finite set. Then, we can write $T^{*}\afs^{n}_{i}=\spec\Lambda[\coor(i)]$. We note that $\bigl(\tfrac{X_j}{X_i}\bigr)_{j\in\{0,\dots,n\}\setminus\{i\}}$ is taken to be a local coordinate for $\afs^{n}_{i}$, and $\xi\bigl(\tfrac{X_j}{X_i}\bigr)$ stands for $d\tfrac{X_j}{X_i}$, where $\bigl(\tfrac{X_j}{X_i},d\tfrac{X_j}{X_i}\bigr)_{j\in\{0,\dots,n\}\setminus\{i\}}$ is the natural local coordinate of the cotangent bundle $T^{*}\afs^{n}_{i}$. For $i\neq j\in\{0,\dots,n\}$, we may compute $\xi\bigl(\frac{X_{j}}{X_{i}}\bigr)=-\frac{X_{j}}{X_{i}}\sum_{k\in\coor(i)}\frac{X_{k}}{X_{i}}\xi\bigl(\frac{X_{k}}{X_{i}}\bigr)$. By formally taking $i=j$, we put
\begin{equation*}
	\xi\bigl(\tfrac{X_{i}}{X_{i}}\bigr)=-\sum_{k\in\coor(i)}\tfrac{X_{k}}{X_{i}}\xi\bigl(\tfrac{X_{k}}{X_{i}}\bigr).
\end{equation*}
Let $\theta\subset\{1,\dots,n\}$ be a non-empty subset, and let $A=[a_{ij}]_{0\leq i\leq d;1\leq j\leq n}$ be a matrix whose entries are integers. We denote by $A[\theta]=[a_{ij}]_{0\leq i\leq d;j\in\theta}$. For every $0\leq i\leq n$, we define subsets of $\Lambda[\coor(i)]$ by
\begin{equation*}
\begin{aligned}
	\Xi_{i}(A,\theta)&=\bigl\{\xi\bigl(\tfrac{X_{j}}{X_{i}}\bigr)\bigm|j\in\{0,\dots,n\}\setminus(\theta\cup\{i\})\bigr\},\\
	\mathrm{L}_{i}(A,\theta)&=\biggl\{\sum_{j\in\theta}a_{kj}\tfrac{X_{j}}{X_{i}}\xi\bigl(\tfrac{X_{j}}{X_{i}}\bigr)\biggm|0\leq k\leq d\biggr\},\\
	\lbox_{i}(A,\theta)&=\biggl\{\prod_{w_{j}>0}\xi\bigl(\tfrac{X_{j}}{X_{i}}\bigr)^{w_{j}}-\prod_{w_{j}<0}\xi\bigl(\tfrac{X_{j}}{X_{i}}\bigr)^{-w_{j}}\biggm|\forall k\in\{0,\dots,d\},\sum_{j\in\theta}a_{kj}w_{j}=0\biggr\}.
\end{aligned}
\end{equation*}
For a subset $\varsigma\subset\Lambda[\coor(i)]$, we denote by $\langle\varsigma\rangle$ the ideal of $\Lambda[\coor(i)]$ generated by $\varsigma$. Let
\begin{equation*}
	\mathfrak{S}_{i}(A,\theta)=\bigl\langle\Xi_{i}(A,\theta)\cup\mathrm{L}_{i}(A,\theta)\cup\lbox_{i}(A,\theta)\bigr\rangle,
\end{equation*}
and let $S_{i}(A,\theta)$ be the closed subset of $T^{*}\afs^{n}_{i}$ defined by $\mathfrak{S}_{i}(A,\theta)$. We note that $S_{i}(A,\theta)$ is conical. Furthermore, for every $1\leq i\leq n$, we define a subset of $\Lambda[\coor(i)]$ by
\begin{equation*}
	\Xi^{\infty}_{i}(A,\theta)=\bigl(\Xi_{i}(A,\theta)\setminus\bigl\{\xi\bigl(\tfrac{X_{0}}{X_{i}}\bigr)\bigr\}\bigr)\cup\bigl\{\tfrac{X_{0}}{X_{i}}\bigr\}.
\end{equation*}
Let $\mathfrak{S}^{\infty}_{i}(A,\theta)=\bigl\langle\Xi^{\infty}_{i}(A,\theta)\cup\mathrm{L}_{i}(A,\theta)\cup\lbox_{i}(A,\theta)\bigr\rangle$, and let $S^{\infty}_{i}(A,\theta)$ be the closed subset of $T^{*}\afs^{n}_{i}$ defined by $\mathfrak{S}^{\infty}_{i}(A,\theta)$. We note that $S^{\infty}_{i}(A,\theta)$ is conical. For a matrix $M$ whose entries are integers, we denote by $M_{\Lambda}$ the image of $M$ by $\mathbb{Z}\rightarrow\Lambda$.

\begin{prp}\label{dimention-local-dimension}
	Let $A=[a_{ij}]_{0\leq i\leq d;1\leq j\leq n}$ be a non-confluent matrix, and let $\theta\subset\{1,\dots,n\}$ be a non-empty subset. Then, for $0\leq k\leq n$, the following statements are true.
	\begin{enumerate}
		\item Set $\Lambda=\rsf$, we have $n\leq\dim S_{k}(A,\theta)\leq n+\rank A[\theta]-\rank_{\rsf}A[\theta]_{\rsf}$.
		\item Set $\Lambda=\frf$, we have $\dim S_{k}(A,\theta)=n$.
	\end{enumerate}
	At the same time, for $1\leq k\leq n$, the following statements are true.
	\begin{enumerate}[start=3]
		\item Set $\Lambda=\rsf$, we have $n\leq\dim S^{\infty}_{k}(A,\theta)\leq n+\rank A[\theta]-\rank_{\rsf}A[\theta]_{\rsf}$.
		\item Set $\Lambda=\frf$, we have $\dim S^{\infty}_{k}(A,\theta)=n$.
	\end{enumerate}
\end{prp}
\begin{proof}
	We first prove 1. Write $r=\rank A[\theta]$ and $r_{p}=\rank_{\rsf}A[\theta]$. There exists $P\in\gl(\mathbb{Q},d+1)$ whose entries are elements of $\mathbb{Z}_{(p)}$, such that the following conditions hold.
	\begin{enumerate}[label=$\star$]
		\item $P_{\rsf}\in\gl(\rsf,d+1)$.
		\item Write $PA[\theta]=[a'_{ij}]_{0\leq i\leq d;j\in\theta}$, then there exist integers $n_{1}<\dots<n_{r_{p}}$ in $\theta$, such that
		\begin{equation*}
			a'_{ij}=\begin{cases}
				1&0\leq i\leq r_{p}-1,j=n_{i+1}\\
				0&0\leq i\leq d,j=n_{l+1},l\in\{0,\dots,r_{p}-1\}\setminus\{i\}
			\end{cases}.
		\end{equation*}
	\end{enumerate}
	This implies that $PA[\theta]_{\rsf}$ is in reduced echelon form. Then, there exists $P'\in\gl(\mathbb{Q},d+1)$, such that the following condition holds
	\begin{enumerate}[label=$\star$]
		\item Write $P'PA[\theta]=[a''_{ij}]_{0\leq i\leq d;j\in\theta}$, then there exist integers $n_{r_{p}+1},\dots,n_{r}$ in $\theta\setminus\{n_{1},\dots,n_{r_{p}}\}$, such that
		\begin{equation*}
			a''_{ij}=\begin{cases}
				1&0\leq i\leq r-1,j=n_{i+1}\\
				0&0\leq i\leq r-1,j=n_{l+1},l\in\{0,\dots,r-1\}\setminus\{i\}\\
				0&r\leq i\leq d,j\in\theta
			\end{cases}.
		\end{equation*}
	\end{enumerate}
	This means that $P'PA[\theta]$ is in reduced echelon form. Let $\qtup(\theta,k)\subset(\mathbb{Z}_{\geq 0})^{4\cdot\#(\theta\setminus\{k\})}$ be the set of $4\cdot\#(\theta\setminus\{k\})$-tuples $(\alpha_{i},\beta_{i},\alpha'_{i},\beta'_{i})_{i\in\theta\setminus\{k\}}$ satisfying the following conditions.
	\begin{enumerate}[label=$\bullet$]
		\item $\sum_{i\in\theta\setminus\{k\}}\beta_{i}=\sum_{i\in\theta\setminus\{k\}}\beta'_{i}$.
		\item If $\alpha_{i}=\alpha'_{i}=0$ for all $i\in\theta\setminus\{k\}$, then the following conditions hold.
		\begin{enumerate}[label=$\circ$]
			\item For any $i\in\theta\setminus\{k,n_{1},\dots,n_{r}\}$, we have $\beta_{i}=0$.
			\item There exists $i\in\theta\setminus\{k,n_{1},\dots,n_{r}\}$, such that $\beta'_{i}>0$.
		\end{enumerate}
		\item If there exists $i\in\theta\setminus\{k\}$ such that $\alpha_{i}>0$, then the following conditions hold.
		\begin{enumerate}[label=$\circ$]
			\item For any $i\in\theta\setminus\{k\}$, we have $\alpha'_{i}=0$.
			\item For any $i\in\theta\setminus\{k,n_{1},\dots,n_{r}\}$, we have $\alpha_{i}=\beta_{i}$.
			\item For any $i\in\{n_{1},\dots,n_{r}\}\setminus\{k\}$, we have $\alpha_{i}\leq\beta_{i}$.
			\item For $i\in\{n_{1},\dots,n_{r}\}\setminus\{k\}$, if $\alpha_{i}<\beta_{i}$, then $\beta'_{i}=0$.
			\item There exists $i\in\theta\setminus\{k,n_{1},\dots,n_{r}\}$, such that $\beta'_{i}>0$.
		\end{enumerate}
		\item If there exists $i\in\theta\setminus\{k\}$ such that $\alpha'_{i}>0$, then the following conditions hold.
		\begin{enumerate}[label=$\circ$]
			\item For any $i\in\theta\setminus\{k\}$, we have $\alpha_{i}=0$.
			\item For any $i\in\theta\setminus\{k,n_{1},\dots,n_{r}\}$, we have $\beta_{i}=0$.
			\item There exists $i\in\theta\setminus\{k,n_{1},\dots,n_{r}\}$, such that $\beta'_{i}>0$.
		\end{enumerate}
	\end{enumerate}
	Then, we denote by $\mono(k)$ the multiplicative monoid generated by $\coor(k)$ and define subsets of $\mono(k)\times\mono(k)$ by
	\begin{equation*}
	\begin{aligned}
		\varrho_{0}&=\biggl\{\biggl(1,\prod_{i\in\theta\setminus\{k\}}\bigl(\tfrac{X_{i}}{X_{k}}\bigr)\biggr)\biggm|\exists i\in\theta\setminus\{k\},\alpha_{i}+\beta_{i}>0\biggr\},\\
		\varrho_{1}&=\bigl\{\bigl(\tfrac{X_{i}}{X_{k}}\xi\bigl(\tfrac{X_{i}}{X_{k}}\bigr),\tfrac{X_{j}}{X_{k}}\xi\bigl(\tfrac{X_{j}}{X_{k}}\bigr)\bigr)\bigm|i\in\theta\setminus\{k,n_{1},\dots,n_{r_{p}}\};j\in\{n_{1},\dots,n_{r_{p}}\}\setminus\{k\}\bigr\},\\
		\varrho_{2}&=\biggl\{\biggl(\prod_{i\in\theta\setminus\{k\}}\bigl(\tfrac{X_{i}}{X_{k}}\bigr)^{\alpha_{i}}\xi\bigl(\tfrac{X_{i}}{X_{k}}\bigr)^{\beta_{i}},\prod_{i\in\theta\setminus\{k\}}\bigl(\tfrac{X_{i}}{X_{k}}\bigr)^{\alpha'_{i}}\xi\bigl(\tfrac{X_{i}}{X_{k}}\bigr)^{\beta'_{i}}\biggr)\biggm|(\alpha_{i},\beta_{i},\alpha'_{i},\beta'_{i})\in\qtup(\theta,k)\biggr\}.
	\end{aligned}
	\end{equation*}
	For a subset $\varrho\subset\mono(k)\times\mono(k)$, we denote by $\varrho^{\rev}$ the subset of $\mono(k)\times\mono(i)$ consisting of the reverse pairs of elements in $\varrho$. We denote by $\orb(\varrho)$ the orbit of $\varrho$ by the diagonal action of $\mono(k)$ on $\mono(k)\times\mono(k)$. We note that if
	\begin{equation*}
		\biggl(\prod_{i\in\theta\setminus\{k\}}\bigl(\tfrac{X_{i}}{X_{k}}\bigr)^{\alpha_{i}}\xi\bigl(\tfrac{X_{i}}{X_{k}}\bigr)^{\beta_{i}},\prod_{i\in\theta\setminus\{k\}}\bigl(\tfrac{X_{i}}{X_{k}}\bigr)^{\alpha'_{i}}\xi\bigl(\tfrac{X_{i}}{X_{k}}\bigr)^{\beta'_{i}}\biggr)\in\orb(\varrho_{0}^{\rev}),
	\end{equation*}
	then $\beta_{i}\geq\beta'_{i}$ for all $i\in\theta\setminus\{k\}$, which implies that
	\begin{equation*}
		\orb(\varrho_{0}^{\rev})\cap\orb(\varrho_{1})=\orb(\varrho_{0}^{\rev})\cap\orb(\varrho_{2})=\varnothing.
	\end{equation*}
	We also note that if $(\alpha_{i},\beta_{i},\alpha'_{i},\beta'_{i})_{i\in\theta\setminus\{k\}}\in\quad(\theta,k)$, then either $\alpha_{i}\geq\alpha'_{i}$ for any $i\in\theta\setminus\{k\}$ or $\alpha_{i}\leq\alpha'_{i}$ for any $i\in\theta\setminus\{k\}$, which implies that
	\begin{equation*}
		\orb(\varrho_{1}^{\rev})\cap\orb(\varrho_{2})=\varnothing.
	\end{equation*}
	Therefore, there exists a term order of $\mono(k)$ defined by $\varrho_{0},\varrho_{1},\varrho_{2}$. We denote this term order by $\prec$. For an ideal $\mathfrak{I}$ of $\rsf[\coor(k)]$, we denote by $\init_{\prec}\mathfrak{I}$ the initial ideal for $\mathfrak{I}$ corresponding to the term order $\prec$. Then, we note that
	\begin{equation*}
		\init_{\prec}\mathfrak{S}_{k}(A,\theta)\subset\bigl\langle\bigl\{\tfrac{X_{i}}{X_{k}},\xi\bigl(\tfrac{X_{j}}{X_{k}}\bigr)\bigm|i\in\{n_{1},\dots,n_{r_{p}}\}\setminus\{k\};j\in\{0,\dots,n\}\setminus\{k,n_{1},\dots,n_{r}\}\bigr\}\bigr\rangle,
	\end{equation*}
	which implies that $\dim S_{k}(A,\theta)\geq\begin{cases}n+r-r_{p}&k\in\theta\setminus\{n_{r_{p}+1},\dots,n_{r}\}\\n+r-r_{p}-1&k\in\{n_{r_{p}+1},\dots,n_{r}\}\end{cases}$. On the other hand, for a subset $\varsigma\subset\coor(k)$, if we assume that $\init_{\prec}\mathfrak{S}_{k}(A,\theta)\subset\langle\varsigma\rangle$, then the following conditions are necessarily satisfied.
	\begin{enumerate}[label=$\star$]
		\item $\Xi_{k}(A,\theta)\subset\varsigma$.
		\item For every $i\in\{n_{1},\dots,n_{r_{p}}\}\setminus\{k\}$, at least one of $\tfrac{X_{i}}{X_{k}},\xi\bigl(\tfrac{X_{i}}{X_{k}}\bigr)$ is an element of $\varsigma$.
		\item $\bigl\{\xi\bigl(\tfrac{X_{i}}{X_{k}}\bigr)\bigm|i\in\theta\setminus\{k,n_{1},\dots,n_{r}\}\bigr\}\subset\varsigma$.
	\end{enumerate}
	This implies that $\#\varsigma\geq\begin{cases}n+r-r_{p}&k\in\theta\setminus\{n_{r_{p}+1},\dots,n_{r}\}\\n+r-r_{p}-1&k\in\{n_{r_{p}+1},\dots,n_{r}\}\end{cases}$, and hence
	\begin{equation*}
		\dim S_{k}(A,\theta)=\begin{cases}
				n+r-r_{p}&k\in\theta\setminus\{n_{r_{p}+1},\dots,n_{r}\}\\
				n+r-r_{p}-1&k\in\{n_{r_{p}+1},\dots,n_{r}\}
		\end{cases}.
	\end{equation*}
	We note that if $k\in\{n_{r_{p}+1},\dots,n_{r}\}$, then $r>r_{p}$, which implies that $n+r-r_{p}-1\geq n$. Therefore, the assertion follows.
	
	Next, we set $r_{p}=r$ and replace $P_{\rsf}$ (resp. $PA[\theta]_{\rsf}$) with $P$ (resp. $PA[\theta]$). Then, repeat the above process, the second assertion follows.
	
	Finally, we replace $\xi\bigl(\frac{X_{0}}{X_{k}}\bigr)$ with $\frac{X_{0}}{X_{k}}$ and repeat the above discussion, then the third and fourth assertions follow.
\end{proof}

We note that $S_{0}(A,\theta),\dots,S_{n}(A,\theta)$ glue together to be a closed subset of $T^{*}\prs^{n}$ which we denote by $\overline{S}(A,\theta)$, and $S_{1}^{\infty}(A,\theta),\dots,S_{n}^{\infty}(A,\theta)$ glue together to be a closed subset of $T^{*}\prs^{n}$ which we denote by $S^{\infty}(A,\theta)$. Then the following is a corollary of \Cref{dimention-local-dimension}.

\begin{cor}\label{dimension-global-dimension}
	Let $A=[a_{ij}]_{0\leq i\leq d;1\leq j\leq n}$ be a non-confluent matrix, and let $\theta\subset\{1,\dots,n\}$ be a non-empty subset. Then, the following statements are true.
	\begin{enumerate}
		\item When $\Lambda=\rsf$, we have $n\leq\dim\overline{S}(A,\theta),\,\dim S^{\infty}(A,\theta)\leq n+\rank A[\theta]-\rank_{\kappa}A[\theta]$.
		\item When $\Lambda=\frf$, we have $\dim\overline{S}(A,\theta)=\dim S^{\infty}(A,\theta)=n$.\qed
	\end{enumerate}
\end{cor}

\subsection{Direct images of conormal bundles}\label{image}

We prove that the union of closed conical subsets defined in \Cref{dimension} is the direct image of the singular support of an $\ell$-adic Kummer sheaf.

Let $\Sigma$ be a generable set of cones in $\mathbb{Z}^{d}$, and let $\varepsilon\subset\edge(\Sigma)$ be a subset. Let
\begin{equation*}
\begin{aligned}
	D(\Sigma,\varepsilon)&=\begin{cases}
		\trv(\Sigma)&\varepsilon=\varnothing\\
		\bigcap_{\epsilon\in\varepsilon}D(\Sigma,\epsilon)&\varepsilon\neq\varnothing
	\end{cases}\\
	\mathcal{E}(\Sigma)&=\{\varepsilon\subset\edge(\Sigma)\mid D(\Sigma,\varepsilon)\neq\varnothing\}
\end{aligned}
\end{equation*}
We note that $\varepsilon\in\mathcal{E}(\Sigma)$ if and only if there exists $\sigma\in\Sigma$, such that $\varepsilon\subset\edge(\sigma)$. For $\varepsilon\in\mathcal{E}(\Sigma)$, we write $\infty_{D(\Sigma,\varepsilon)}=D(\Sigma,\varepsilon)\times\infty$. For a matrix $B=[b_{ij}]_{1\leq i\leq d;1\leq j\leq n}$ whose entries are integers, let
\begin{equation*}
\begin{aligned}
	D_{B}(\Sigma,\varepsilon)&=D_{B}(\Sigma)\cap\prs^{n}_{D(\Sigma,\varepsilon)},\\
	D_{B}^{\infty}(\Sigma,\varepsilon)&=D_{B}(\Sigma)\cap\infty_{D(\Sigma,\varepsilon)}.
\end{aligned}
\end{equation*}
Assume that $B$ is $\Sigma$-good, in which case $\overline{D}_{B}(\Sigma)$ is a divisor with simple normal crossings. Let
\begin{equation*}
	N_{B}(\Sigma)=\bigcup_{\varepsilon\subset\mathcal{E}(\Sigma)}(T^{*}_{\prs^{n}_{D(\Sigma,\epsilon)}}\prs^{n}_{\trv(\Sigma)}\cup T^{*}_{\infty_{D(\Sigma,\epsilon)}}\prs^{n}_{\trv(\Sigma)}\cup T^{*}_{D_{B}(\Sigma,\epsilon)}\prs^{n}_{\trv(\Sigma)}\cup T^{*}_{D_{B}^{\infty}(\Sigma,\epsilon)}\prs^{n}_{\trv(\Sigma)}).
\end{equation*}
Let $\overline{j}_{B,\Sigma}:U_{B}\rightarrow\prs^{n}_{\trv(\Sigma)}$ be the open immersion.

\begin{prp}\label{image-kummer-sheaf}
	Assume that $\Lambda$ is a field. Let $\Sigma$ be a generable set of cones in $\mathbb{Z}^{d}$, and let $B=[b_{ij}]_{1\leq i\leq d;1\leq j\leq n}$ be a $\Sigma$-good matrix. Let $\chi=(\chi_{0},\dots,\chi_{d}):\mu_{q-1}^{d+1}\rightarrow\ql^{*}$ be a multiplicative character, then the following statements are true.
	\begin{enumerate}
		\item $SS\overline{j}_{B,\Sigma!}\kum_{\chi}(B)=N_{B}(\Sigma)$.
		\item $CC\overline{j}_{B,\Sigma!}\kum_{\chi}(B)=(-1)^{d+n}[N_{B}(\Sigma)]$.
	\end{enumerate}
\end{prp}
\begin{proof}
	By \Cref{divisor-good-is-normal}, we note that $\overline{D}_{B}(\Sigma)=\prs^{n}_{\trv(\Sigma)}\setminus U_{B}$ is a divisor with simple normal crossings. We also note that $\kum_{\chi}(B)$ is a lisse sheaf which is tamely ramified along $\overline{D}_{B}(\Sigma)$. Then, the first assertion follows from \cite[Proposition~4.11]{saito2017characteristic}, and the second assertion follows from \cite[Theorem~7.14]{saito2017characteristic}.
\end{proof}

Next, let $\overline{\pi}_{\Sigma}:\prs^{n}_{\trv(\Sigma)}\rightarrow\prs^{n}$ be the projection. We denote by $d_{\Sigma}:\prs^{n}_{\trv(\Sigma)}\times_{\prs^{n}}T^{*}\prs^{n}\rightarrow T^{*}\prs^{n}_{\trv(\Sigma)}$ the morphism of relative differentials and by $\varpi_{\Sigma}:\prs^{n}_{\trv(\Sigma)}\times_{\prs^{n}}T^{*}\prs^{n}\rightarrow T^{*}\prs^{n}$ be the projection. Assume that $\Delta(\Sigma)$ is a complete fan, in which case we say that $\Sigma$ is \textit{complete}. Then, both $\overline{\pi}_{\Sigma}$ and $\varpi_{\Sigma}$ are projective and hence proper. For a closed subset $N\subset T^{*}\prs^{n}_{\trv(\Sigma)}$, we denote by
\begin{equation*}
	\overline{\pi}_{\Sigma\circ}N=\varpi_{\Sigma}(d_{\Sigma}^{-1}N),
\end{equation*}
which is a closed subset of $T^{*}\prs^{n}$. Furthermore, we denote by
\begin{equation}\label{image-direct-image}
	\overline{\pi}_{\Sigma!}=\varpi_{\Sigma*}\comp d_{\Sigma}^{!}:CH_{d+n}(N)\rightarrow CH_{n}(\overline{\pi}_{\Sigma\circ}N).
\end{equation}
If every irreducible component of $N$ (resp. $\overline{\pi}_{\Sigma\circ}N$) is of dimension $\leq d+n$ (resp. $\leq n$), then \cref{image-direct-image} defines a morphism $\overline{\pi}_{\Sigma!}:Z_{d+n}(N)\rightarrow Z_{n}(\overline{\pi}_{\Sigma\circ}N)$.

\begin{lem}\label{image-direct-image-easy}
	Let $\Sigma$ be a generable set of cones in $\mathbb{Z}^{d}$ which is complete, and let $\varepsilon\in\mathcal{E}(\Sigma)$. If $\Lambda$ is a field, then the following statements are true.
	\begin{enumerate}
		\item $\overline{\pi}_{\Sigma\circ}(T^{*}_{\prs^{n}_{D(\Sigma,\varepsilon)}}\prs^{n}_{\trv(\Sigma)})=T^{*}_{\prs^{n}}\prs^{n}$.
		\item $\overline{\pi}_{\Sigma\circ}(T^{*}_{\infty_{D(\Sigma,\varepsilon)}}\prs^{n}_{\trv(\Sigma)})=T^{*}_{\infty}\prs^{n}$.
	\end{enumerate}
\end{lem}
\begin{proof}
	We first prove 1. Note that $T^{*}_{\prs^{n}_{D(\Sigma,\varepsilon)}}\prs^{n}_{\trv(\Sigma)}=T^{*}_{D(\Sigma,\varepsilon)}\trv(\Sigma)\times T^{*}_{\prs^{n}}\prs^{n}$. Therefore, the inverse image $d_{\Sigma}^{-1}(T^{*}_{\prs^{n}_{D(\Sigma,\varepsilon)}}\prs^{n}_{\trv(\Sigma)})=D(\Sigma,\varepsilon)\times T^{*}_{\prs^{n}}\prs^{n}$. Then, the assertion follows.
	
	Then, we prove 2. Note that $T^{*}_{\infty_{D(\Sigma,\varepsilon)}}\prs^{n}_{\trv(\Sigma)}=T^{*}_{D(\Sigma,\varepsilon)}\trv(\Sigma)\times T^{*}_{\infty}\prs^{n}$. Therefore, the inverse image $d_{\Sigma}^{-1}(T^{*}_{\infty_{D(\Sigma,\varepsilon)}}\prs^{n}_{\trv(\Sigma)})=D(\Sigma,\varepsilon)\times T^{*}_{\infty}\prs^{n}$. Then, the assertion follows.
\end{proof}

For a generable set $\Sigma$ of cones in $\mathbb{Z}^{d}$ and $\varepsilon\in\mathcal{E}(\Sigma)$, we denote by $\Sigma(\varepsilon)=\bigcup_{\epsilon\in\varepsilon}\Sigma(\epsilon)$. Note that $\sigma\in\Sigma(\varepsilon)$ if and only if $\varepsilon\cap\edge(\sigma)\neq\varnothing$. For a matrix $B=[b_{ij}]_{1\leq i\leq d;1\leq j\leq n}$ whose entries are integers, we define a subset of $\{1,\dots,n\}$ by
\begin{equation*}
	\theta_{B}(\varepsilon)=\begin{cases}
		\{1,\dots,n\}&\varepsilon=\varnothing\\
		\bigcap_{\epsilon\in\varepsilon}\theta_{B}(\epsilon)&\varepsilon\neq\varnothing
	\end{cases}
\end{equation*}
where $\theta_{B}(\epsilon)=\{1\leq i\leq n\mid b_{\epsilon i}=\min\{b_{\epsilon j}\mid 1\leq j\leq n\}\}$ for $\epsilon\in\varepsilon$.

\begin{lem}\label{image-direct-image-hard}
	Let $\Sigma$ be a generable set of cones in $\mathbb{Z}^{d}$ which is complete, and let $\varepsilon\in\mathcal{E}(\Sigma)$. Let $B=[b_{ij}]_{1\leq i\leq d;1\leq j\leq n}$ be a $\Sigma$-good matrix. If $\Lambda$ is a field, then the following statements are true.
	\begin{enumerate}
		\item $\overline{\pi}_{\Sigma\circ}(T^{*}_{D_{B}(\Sigma,\varepsilon)}\prs^{n}_{\trv(\Sigma)})\subset\bigcup_{\sigma\in\Sigma(\varepsilon)}\overline{S}(\widehat{B},\theta_{B}(\varepsilon\cap\edge(\sigma)))$.
		\item $\overline{\pi}_{\Sigma\circ}(T^{*}_{D_{B}^{\infty}(\Sigma,\varepsilon)}\prs^{n}_{\trv(\Sigma)})\subset\bigcup_{\sigma\in\Sigma(\varepsilon)}S^{\infty}(\widehat{B},\theta_{B}(\varepsilon\cap\edge(\sigma)))$.
	\end{enumerate}
\end{lem}
\begin{proof}
	We first prove 1. For $0\leq k\leq n$ and $\sigma\in\Sigma(\varepsilon)$, consider the following commutative diagram over $\Lambda$.
	\begin{equation*}
	\begin{tikzcd}
		T^{*}\afs^{n}_{k,\trv(\sigma)}\arrow[d]\arrow[dr,phantom,"\square"]&\afs^{n}_{k,\trv(\sigma)}\times_{\afs^{n}_{k}}T^{*}\afs^{n}_{k}\arrow[d]\arrow[l,"d_{k,\sigma}"']\arrow[r,"\varpi_{k,\sigma}"]&T^{*}\afs^{n}_{k}\arrow[d]\\
		T^{*}\prs^{n}_{\trv(\Sigma)}&\prs^{n}_{\trv(\Sigma)}\times_{\prs^{n}}T^{*}\prs^{n}\arrow[l,"d_{\Sigma}"']\arrow[r,"\varpi_{\Sigma}"]\arrow[ur,phantom,"\square"]&T^{*}\prs^{n}
	\end{tikzcd}.
	\end{equation*}
	We write $T^{*}\trv(\sigma)=\spec\Lambda[\Gamma(\sigma)][\xi(t_{\sigma,\epsilon})\mid\epsilon\in\edge(\sigma)]$. Let $\theta=\theta_{B}(\varepsilon\cap\edge(\sigma))$, and let $D_{k,B}(\sigma,\varepsilon)=D_{B}(\sigma,\varepsilon)\cap\afs^n_{k,\trv(\sigma)}$. Then, we note that $D_{k,B}(\sigma,\varepsilon)$ is the closed subset of $\afs^n_{k,\trv(\sigma)}$ defined by the ideal generated by the set
		\begin{equation*}
			\biggl\{\sum_{i\in\theta\setminus\{k\}}\tfrac{X_{i}}{X_{k}}\prod_{\epsilon\in\edge(\sigma)\setminus\varepsilon}t_{\sigma,\epsilon}^{b_{\epsilon j}}\biggr\}\cup\{t_{\sigma,\epsilon}\mid\epsilon\in\varepsilon\}.
		\end{equation*}
	As a consequence, the conormal bundle $T^{*}_{D_{k,B}(\sigma,\varepsilon)}\afs^{n}_{k,\trv(\sigma)}$ is the closed subset of $T^{*}\afs^{n}_{\trv(\sigma)}$ defined by the ideal $\mathfrak{I}(T^{*}_{D_{k,B}(\sigma,\varepsilon)}\afs^{n}_{k,\trv(\sigma)})$ of $\Lambda[\Gamma(\sigma)][\xi(t_{\sigma,\epsilon})\mid\epsilon\in\edge(\sigma)][\coor(k)]$ generated by
	\begin{equation*}
	\begin{aligned}
		\Xi_{k}(\widehat{B},\theta)&\cup\biggl\{\sum_{i\in\theta\setminus\{k\}}\tfrac{X_{i}}{X_{k}}\prod_{\epsilon\in\edge(\sigma)\setminus\varepsilon}t_{\sigma,\epsilon}^{b_{\epsilon j}}\biggr\}\cup\{t_{\sigma,\epsilon}\mid\epsilon\in\varepsilon\}\cup\biggl\{\xi\bigl(\tfrac{X_{i}}{X_{k}}\bigr)-\prod_{\epsilon\in\edge(\sigma)\setminus\varepsilon}t_{\sigma,\epsilon}^{b_{\epsilon i}}\biggm|i\in\theta\biggr\}\\
		&\cup\biggl\{\xi(t_{\sigma,\epsilon})-\tfrac{1}{t_{\sigma,\epsilon}}\sum_{i\in\theta\setminus\{k\}}b_{\epsilon i}\tfrac{X_{i}}{X_{k}}\prod_{\epsilon'\in\edge(\sigma)\setminus\varepsilon}t_{\sigma,\epsilon'}^{b_{\epsilon i}}\biggm|\epsilon\in\edge(\sigma)\setminus\varepsilon\biggr\}.
	\end{aligned}
	\end{equation*}
	Then, the defining ideal $\mathfrak{I}(d_{k,\sigma}^{-1}(T^{*}_{D_{k,B}(\sigma,\varepsilon)}\afs^{n}_{k,\trv(\sigma)}))$ of the inverse image $d_{k,\sigma}^{-1}(T^{*}_{D_{k,B}(\sigma,\varepsilon)}\afs^{n}_{k,\trv(\sigma)})$ as a closed subset of $\afs^{n}_{k,\trv(\sigma)}\times_{\afs^{n}_{k}}T^{*}\afs^{n}_{k}$ is the extension of $\mathfrak{I}(T^{*}_{D_{k,B}(\sigma,\varepsilon)}\afs^{n}_{k,\trv(\sigma)})$ by the homomorphism of $\Lambda[\Gamma(\sigma)][\coor(k)]$-algebras
	\begin{equation*}
	\begin{aligned}
		\Lambda[\Gamma(\sigma)][\xi(t_{\sigma,\epsilon})\mid\epsilon\in\edge(\sigma)][\coor(k)]&\rightarrow\Lambda[\Gamma(\sigma)][\coor(k)]\\
		\xi(t_{\sigma,\epsilon})&\mapsto 0
	\end{aligned}
	\end{equation*}
	which is equal to the ideal of $\Lambda[\Gamma(\sigma)][\coor(k)]$ generated by
	\begin{equation*}
	\begin{aligned}
		\Xi_{k}(\widehat{B},\theta)&\cup\biggl\{\sum_{i\in\theta\setminus\{k\}}\tfrac{X_{i}}{X_{k}}\prod_{\epsilon\in\edge(\sigma)}t_{\sigma,\epsilon}^{b_{\epsilon j}}\biggr\}\cup\{t_{\sigma,\epsilon}\mid\epsilon\in\varepsilon\}\cup\biggl\{\xi\bigl(\tfrac{X_{i}}{X_{k}}\bigr)-\prod_{\epsilon\in\edge(\sigma)\setminus\varepsilon}t_{\sigma,\epsilon}^{b_{\epsilon i}}\biggm|i\in\theta\biggr\}\\
		&\cup\biggl\{\tfrac{1}{t_{\sigma,\epsilon}}\sum_{i\in\theta\setminus\{k\}}b_{\epsilon i}\tfrac{X_{i}}{X_{k}}\prod_{\epsilon'\in\edge(\sigma)\setminus\varepsilon}t_{\sigma,\epsilon'}^{b_{\epsilon i}}\biggm|\epsilon\in\edge(\sigma)\setminus\varepsilon\biggr\}.
	\end{aligned}
	\end{equation*}
	We next note that the defining ideal $\mathfrak{I}(\varpi_{k,\sigma}d_{k,\sigma}^{-1}(T^{*}_{D_{k,B}(\sigma,\varepsilon)}\afs^{n}_{k,\trv(\sigma)}))$ of the closure of the image $\varpi_{k,\sigma}d_{k,\sigma}^{-1}(T^{*}_{D_{k,B}(\sigma,\varepsilon)}\afs^{n}_{k,\trv(\sigma)})$ is the contraction of $\mathfrak{I}(d_{k,\sigma}^{-1}(T^{*}_{D_{k,B}(\sigma,\varepsilon)}\afs^{n}_{k,\trv(\sigma)}))$ by the embedding
	\begin{equation*}
		\Lambda[\coor(k)]\rightarrow\Lambda[\Gamma(\sigma)][\coor(k)].
	\end{equation*}
	We also note that the fact that
	\begin{equation*}
		\sum_{i\in\theta\setminus\{k\}}b_{\epsilon i}\tfrac{X_{i}}{X_{k}}\xi\bigl(\tfrac{X_{i}}{X_{k}}\bigr)=\sum_{i\in\theta\setminus\{k\}}b_{\epsilon i}\tfrac{X_{i}}{X_{k}}\left\{\xi\bigl(\tfrac{X_{i}}{X_{k}}\bigr)-\prod_{\epsilon\in\edge(\sigma)\setminus\varepsilon}t_{\sigma,\epsilon}^{b_{\epsilon i}}\right\}+\sum_{i\in\theta\setminus\{k\}}b_{\epsilon i}\tfrac{X_{i}}{X_{k}}\prod_{\epsilon'\in\edge(\sigma)\setminus\varepsilon}t_{\sigma,\epsilon'}^{b_{\epsilon i}}
	\end{equation*}
	implies that $\sum_{i\in\theta\setminus\{k\}}b_{\epsilon i}\tfrac{X_{i}}{X_{k}}\xi\bigl(\tfrac{X_{i}}{X_{k}}\bigr)\in\mathfrak{I}(d_{k,\sigma}^{-1}(T^{*}_{D_{k,B}(\sigma,\varepsilon)}\afs^{n}_{k,\trv(\sigma)}))$. Therefore, we know that $L_{k}(\widehat{B},\theta)\subset\mathfrak{I}(d_{k,\sigma}^{-1}(T^{*}_{D_{k,B}(\sigma,\varepsilon)}\afs^{n}_{k,\trv(\sigma)}))$. At the same time, if $\sum_{\substack{i\in\theta}}b_{\epsilon i}w_i=0$, then we have
	\begin{equation*}
		\sum_{\substack{i\in\theta\\w_i>0}}b_{\epsilon i}w_i=-\sum_{\substack{i\in\theta\\w_i<0}}b_{\epsilon i}w_i.
	\end{equation*}
	As a consequence, we know that $\prod_{\substack{i\in\theta\\w_i>0}}\left(\prod_{\epsilon'\in\edge(\sigma)\setminus\varepsilon}t_{\sigma,\epsilon'}^{b_{\epsilon i}}\right)^{w_i}=\prod_{\substack{i\in\theta\\w_i<0}}\left(\prod_{\epsilon'\in\edge(\sigma)\setminus\varepsilon}t_{\sigma,\epsilon'}^{b_{\epsilon i}}\right)^{-w_i}$. Then, we have
	\begin{align*}
		\prod_{\substack{i\in\theta\\w_i>0}}\xi\bigl(\tfrac{X_{i}}{X_{k}}\bigr)^{w_i}-\prod_{\substack{i\in\theta\\w_i<0}}\xi\bigl(\tfrac{X_{i}}{X_{k}}\bigr)^{-w_i}=&\left\{\prod_{\substack{i\in\theta\\w_i>0}}\xi\bigl(\tfrac{X_{i}}{X_{k}}\bigr)^{w_i}-\prod_{\substack{i\in\theta\\w_i>0}}\left(\prod_{\epsilon'\in\edge(\sigma)\setminus\varepsilon}t_{\sigma,\epsilon'}^{b_{\epsilon i}}\right)^{w_i}\right\}\\
		&-\left\{\prod_{\substack{i\in\theta\\w_i<0}}\xi\bigl(\tfrac{X_{i}}{X_{k}}\bigr)^{-w_i}-\prod_{\substack{i\in\theta\\w_i<0}}\left(\prod_{\epsilon'\in\edge(\sigma)\setminus\varepsilon}t_{\sigma,\epsilon'}^{b_{\epsilon i}}\right)^{-w_i}\right\}.
	\end{align*}
	Therefore, we know that $\prod_{\substack{i\in\theta\\w_i>0}}\xi\bigl(\tfrac{X_{i}}{X_{k}}\bigr)^{w_i}-\prod_{\substack{i\in\theta\\w_i<0}}\xi\bigl(\tfrac{X_{i}}{X_{k}}\bigr)^{-w_i}\in\mathfrak{I}(d_{k,\sigma}^{-1}(T^{*}_{D_{k,B}(\sigma,\varepsilon)}\afs^{n}_{k,\trv(\sigma)}))$ and hence $\lbox_{k}(\widehat{B},\theta)\subset\mathfrak{I}(d_{k,\sigma}^{-1}(T^{*}_{D_{k,B}(\sigma,\varepsilon)}\afs^{n}_{k,\trv(\sigma)}))$. Finally, we conclude that $L_{k}(\widehat{B},\theta)\cup\lbox_{k}(\widehat{B},\theta)\subset\mathfrak{I}(d_{k,\sigma}^{-1}(T^{*}_{D_{k,B}(\sigma,\varepsilon)}\afs^{n}_{k,\trv(\sigma)}))$ which implies that
	\begin{equation*}
		\mathfrak{S}(\widehat{B},\theta)\subset\mathfrak{I}(\varpi_{k,\sigma}d_{k,\sigma}^{-1}(T^{*}_{D_{k,B}(\sigma,\varepsilon)}\afs^{n}_{k,\trv(\sigma)})).
	\end{equation*}
	Therefore, the first assertion follows.
	
	Consider only the case where $1\leq k\leq n$ and replace $\Xi_{k}(\widehat{B},\theta)$ with $\Xi_{k}^{\infty}(\widehat{B},\theta)$. Repeat the above process, then the second assertion follows.
\end{proof}

\begin{cor}\label{image-direct-image-kummer-sheaf}
	Let $\Sigma$ be a generable set of cones in $\mathbb{Z}^{d}$ which is complete, and let $B=[b_{ij}]_{1\leq i\leq d;1\leq j\leq n}$ be a $\Sigma$-good matrix. If $\Lambda$ is a field, then we have
	\begin{equation*}
		\overline{\pi}_{\Sigma\circ	}N_{B}(\Sigma)\subset T^{*}_{\prs^{n}}\prs^{n}\cup T^{*}_{\infty}\prs^{n}\bigcup_{\substack{\varepsilon\in\mathcal{E}(\Sigma)\\\sigma\in\Sigma(\varepsilon)}}(\overline{S}(\widehat{B},\theta_{B}(\varepsilon\cap\edge(\sigma)))\cup S^{\infty}(\widehat{B},\theta_{B}(\varepsilon\cap\edge(\sigma)))).
	\end{equation*}
\end{cor}
\begin{proof}
	The assertion follows from \Cref{image-direct-image-easy} and \Cref{image-direct-image-hard}.
\end{proof}

\subsection{\texorpdfstring{$p$}{p}-nondegenerate square matrices}\label{square}
We prove a formula for the characteristic cycle of a non-confluent $\ell$-adic GKZ hypergeometric sheaf satisfying the $p$-nondegeneracy condition. We also explain how to reduce a general square case which does not satisfy the condition we need to a case where it is satisfied.

\begin{dfn}
	Let $A=[a_{ij}]_{0\leq i\leq d;1\leq j\leq n}$ be a matrix whose entries are integers.
	\begin{enumerate}
		\item If there exists $B$ which is sub-non-confluent, such that $A=\widehat{B}$, then we say that $A$ is \textnormal{standard non-confluent}.
		\item If there exists $P\in\gl(d+1,\mathbb{Z})$, such that $PA$ is standard non-confluent, then we say that $A$ is \textnormal{non-confluent}.
		\item If for any $\theta\subset\{1,\dots, n\}$, we have $\rank A[\theta]=\rank_{\rsf}A[\theta]_{\rsf}$, then we say that $A$ is \textnormal{$p$-nondegenerate}.
	\end{enumerate}
\end{dfn}

\begin{dfn}
	Let $B=[a_{ij}]_{0\leq i\leq d;1\leq j\leq n}$ be a matrix whose entries are integers. If $\widehat{B}$ is $p$-nondegenerate, then we say that $B$ is \textnormal{sub-$p$-nondegenerate}.
\end{dfn}

Recall that $j:\afs^{n}\rightarrow\prs^{n}$ is the open immersion. For a closed subset $C\subset T^{*}\prs^{n}$, we denote by $j^{*}C=C\times_{T^{*}\prs^{n}}T^{*}\afs^{n}$ which is a closed subset of $T^{*}\afs^{n}$.

\begin{thm}\label{square-characteristic-cycle}
	Let $B=[b_{ij}]_{1\leq i\leq d;1\leq j\leq n}$ be a sub-non-confluent matrix. Let $\chi=(\chi_{0},\dots,\chi_{d}):\mu_{q-1}^{d+1}\rightarrow\ql^{*}$ be a multiplicative character where $\chi_{0}$ is nontrivial, and let $\psi:\rsf\rightarrow\ql^{*}$ be a nontrivial additive character. If $B$ is sub-$p$-nondegenerate and $\Lambda$ is a field, then we have
	\begin{equation*}
		CC\gkz_{\chi}(B)=(-1)^{d+n}j^{*}\overline{\pi}_{\Sigma!}[N_{B}(\Sigma)]\in Z_{n}(T^{*}\afs^{n})
	\end{equation*}
	where $\Sigma$ is any generable set such that $B$ is $\Sigma$-good.
\end{thm}
\begin{proof}
	First, by \Cref{resolution-existence}, there exists a generable set $\Sigma$ of cones in $\mathbb{Z}^{d}$ which is complete, such that $B$ is $\Sigma$-good. We note that
	\begin{equation*}
		\gkz_{\chi}(B)\cong j^*R\overline{\pi}_{\Sigma!}\overline{j}_{B,\Sigma!}\kum_{\chi}(B).
	\end{equation*}
	By \Cref{dimension-global-dimension}, \Cref{image-direct-image-kummer-sheaf} and \Cref{image-kummer-sheaf}.1 we have
	\begin{equation}\label{square-dimension-condition}
		\dim\overline{\pi}_{\Sigma\circ}SSj_{B,\Sigma!}\kum_{\chi}(B)\leq n.
	\end{equation}
	Therefore, by \cite[Lemma~5.13.3]{saito2017characteristic}, \cite[Lemma~5.11.2]{saito2017characteristic} and \cite[Theorem~2.2.5]{saito2021characteristic}, we have
	\begin{equation*}
		CC\gkz_{\chi}(B)=j^{*}\overline{\pi}_{\Sigma!}CC\overline{j}_{B,\Sigma!}\kum_{\chi}(B)
	\end{equation*}
	since \cref{square-dimension-condition} holds. Then, the assertion follows from \Cref{image-kummer-sheaf}.2 In fact, the assertion can be proved by choosing any generable set $\Sigma$ which is complete meeting the requirement that $B$ is $\Sigma$-good, and the results does not depend on the choice of such $\Sigma$.
\end{proof}

\begin{lem}\label{square-p-nondegenerate}
	Let $A=[a_{ij}]_{0\leq i\leq d;1\leq j\leq n}$ be a matrix of rank $d+1$ whose entries are integers. Assume that $A$ is not $p$-nondegenerate. Set $F=\diag\{p,1,\dots,1\}\in\mat(d+1,\mathbb{Z})$. If $n=d+1$, then there exist $P_{1},\dots,P_{r}\in\gl(d+1,\mathbb{Z})$, such that $F^{-1}P_{r}\dots F^{-1}P_{1}A\in\mat(d+1,\mathbb{Z})$ and it is $p$-nondegenerate.
\end{lem}
\begin{proof}
	We note that $A$ is $p$-nondegenerate if and only if the $p$-adic valuation $\mathrm{val}_{p}(\det A)=0$. Then, we argue by induction on $\mathrm{val}_{p}(\det A)$. If $\mathrm{val}_{p}(\det A)=1$, then there exists $P\in\gl(d+1,\mathbb{Z})$, such that $a'_{01}\equiv\dots\equiv a'_{0n}\equiv 0\pmod{p}$ where we write $PA=[a'_{ij}]_{0\leq i\leq d;1\leq j\leq d+1}$, which implies that $F^{-1}PA\in\mat(d+1,\mathbb{Z})$. We note that $\det F^{-1}PA=p^{-1}\det A$, and hence
	\begin{equation*}
		\mathrm{val}_{p}(\det F^{-1}PA)=\mathrm{val}_{p}(\det A)-1=0.
	\end{equation*}
	Next, we assume that the argument is true for $1\leq\mathrm{val}_{p}(\det A)\leq\mathrm{val}$, then we prove that it is also true for $\mathrm{val}_{p}(\det A)=\mathrm{val}+1$. Again, there exists $P_{1}\in\gl(d+1,\mathbb{Z})$, such that $F^{-1}P_{1}A\in\mat(d+1,\mathbb{Z})$ and $\mathrm{val}_{p}(\det F^{-1}P_{1}A)=\mathrm{val}_{p}(\det A)-1$. By the induction hypothesis, if $\mathrm{val}_{p}(\det F^{-1}P_{1}A)>0$, then there exist $P_{2},\dots,P_{r}\in\gl(d+1,\mathbb{Z})$, such that $F^{-1}P_{r}\dots F^{-1}P_{1}A\in\mat(d+1,\mathbb{Z})$ and $\mathrm{val}_{p}(\det F^{-1}P_{r}\dots F^{-1}P_{1}A)=0$. Therefore $P_{1},\dots,P_{r}$ are the required matrices.
\end{proof}

Assume that $n=d+1$. For a matrix $M=[m_{ij}]_{0\leq i\leq d;1\leq j\leq d+1}\in\mat(d+1,\mathbb{Z})$, we define $\varphi_{M}:\trs^{d+1}\rightarrow\trs^{d+1}$ to be the endomorphism induced by the ring endomorphism
\begin{equation*}
	\Lambda[t_{0}^{\pm 1},\dots,t_{d}^{\pm 1}]\rightarrow\Lambda[t_{0}^{\pm 1},\dots,t_{d}^{\pm 1}]:t_{j}\mapsto\prod_{i=0}^{d}t_{i}^{m_{ij}}
\end{equation*}
For a multiplicative character $\chi=(\chi_{0},\dots,\chi_{d}):\mu_{q-1}^{d+1}\rightarrow\ql^{*}$, we denote by
\begin{equation*}
	\chi^{M}=(\chi_{0}^{m_{00}}\dots\chi_{d}^{m_{0d}},\dots,\chi_{0}^{m_{d0}}\dots\chi_{d}^{m_{dd}}).
\end{equation*}

\begin{lem}\label{square-sheaf-isomorphism}
	Let $A=[a_{ij}]_{0\leq i\leq d;1\leq j\leq n}$ be a non-confluent matrix. Let $\chi=(\chi_{0},\dots,\chi_{d}):\mu_{q-1}^{d+1}\rightarrow\ql^{*}$ be a multiplicative character, and let $\psi:\rsf\rightarrow\ql^{*}$ be a nontrivial additive character. If $n=d+1$, then the following statements are true.
	\begin{enumerate}
		\item For $P\in\gl(d+1,\mathbb{Z})$, we have $\hyp_{\psi}(A,\chi)\cong\hyp_{\psi}(PA,\chi^{P})$.
		\item Set $F=\diag\{p,1,\dots,1\}\in\mat(d+1,\mathbb{Z})$, then we have $\hyp_{\psi}(A,\chi)\cong\hyp_{\psi}(FA,\chi)$.
	\end{enumerate}
\end{lem}
\begin{proof}
	\begin{enumerate}
		\item First, we consider the following commutative diagram over $\rsf$.
		\begin{equation*}
		\begin{tikzcd}
			\trs^{d+1}&\afs^{n}_{\trs(\widehat{\delta})}\arrow[d,"\varphi_{P}"]\arrow[l,"\widehat{\tau}"']\arrow[r,"\widehat{\pi}"]\arrow[dd,near end,bend right=60,"\widehat{g}_{PA}"']&\afs^{n}\\
			&\afs^{n}_{\trs^{d+1}}\arrow[d,"\widehat{g}_{A}"]\arrow[ul,crossing over,near end,bend left,"\widehat{\tau}"]\arrow[ur,near end,bend right,"\widehat{\pi}"']&\\
			&\afs^{1}_{\dagger}&
		\end{tikzcd}.
		\end{equation*}
		We note that $\varphi_{P}^{*}(\widehat{\tau}^{*}(\widehat{\tau}_{0}^{*}\mathscr{K}_{\chi_{0}}\otimes\dots\otimes\widehat{\tau}_{d}^{*}\mathscr{K}_{\chi_{d}}))\cong\widehat{\tau}^{*}(\widehat{\tau}_{0}^{*}\mathscr{K}_{\chi_{0}^{P}}\otimes\dots\otimes\widehat{\tau}_{d}^{*}\mathscr{K}_{\chi_{d}^{P}})$. Then, the assertion follows from the projection formula.
		
		\item We note that $\varphi_{F}$ is a partial Frobenius morphism which is a universal homeomorphism and that $\mathscr{K}_{\chi_{i}^{F}}\cong\mathscr{K}_{\chi_{i}}$ for $0\leq i\leq d$. Then, perform the same process as above, the assertion follows.\qedhere
	\end{enumerate}
\end{proof}

\begin{thm}\label{square-p-nondegerate-non-confluent}
	Let $A=[a_{ij}]_{0\leq i\leq d;1\leq j\leq n}$ be a non-confluent matrix. Let $\chi=(\chi_{0},\dots,\chi_{d}):\mu_{q-1}^{d+1}\rightarrow\ql^{*}$ be a multiplicative character, and let $\psi:\rsf\rightarrow\ql^{*}$ be a non-trivial additive character. If $n=d+1$, then there exist a sub-non-confluent matrix $B=[b_{ij}]_{0\leq i\leq d;1\leq j\leq n}$ which is sub-$p$-nondegenerate and a multiplicative character $\widehat{\chi}=(\widehat{\chi}_{0},\dots,\widehat{\chi}_{d}):\mu_{q-1}^{d+1}\rightarrow\ql^{*}$, such that
	\begin{equation*}
		\hyp_{\psi}(A,\chi)\cong\hyp_{\psi}(\widehat{B},\widehat{\chi}).
	\end{equation*}
\end{thm}
\begin{proof}
	Set $F=\diag\{p,1,\dots,1\}\in\mat(d+1,\mathbb{Z})$. By \Cref{square-p-nondegenerate}, there exist $P_{1},\dots,P_{r}\in\gl(d+1,\mathbb{Z})$, such that $F^{-1}P_{r}\dots F^{-1}P_{1}A\in\mat(d+1,\mathbb{Z})$ and it is $p$-nondegenerate. We note that if $A$ is non-confluent, then so is $F^{-1}P_{r}\dots F^{-1}P_{1}A$, which implies the existence of matrices $B$ and $P\in\gl(d+1,\mathbb{Z})$, such that $PF^{-1}P_{r}\dots F^{-1}P_{1}A=\widehat{B}$. Let $\widehat{\chi}=\chi^{P_{1}\dots P_{r}}$. Then, by \Cref{square-sheaf-isomorphism}, we have $\hyp_{\psi}(A,\chi)\cong\hyp_{\psi}(\widehat{B},\widehat{\chi})$.
\end{proof}

\section{Comparison theorems of characteristic cycles}\label{comparison}
We formulate comparison theorems among different types of characteristic cycles. We use the specialization map to explain the relationship between the characteristic cycle of an $\ell$-adic GKZ-type sheaf over a finite extension of $\mathbb{Q}$ and the one of a non-confluent as well as $p$-nondegenerate $\ell$-adic GKZ hypergeometric sheaf in \Cref{specialization}, and we use Riemann-Hilbert correspondence to explain the relationship between an $\ell$-adic GKZ-type sheaf and the corresponding regular GKZ hypergeometric $\mathcal{D}$-module.
\subsection{Comparison via specialization maps}\label{specialization}

We refer to \cite[20.1]{fulton2016intersection} for the definition of relative Chow groups, and refer to \cite[20.3]{fulton2016intersection} for the definition of specialization maps. For a scheme $Y$ over $\intr$, we write $Y_{\Lambda}=Y\times_{\spec\intr}\spec\Lambda$. We denote by $\spmap:CH_{i}(Y_{\frf}/\frf)\rightarrow CH_{i}(Y_{\rsf}/\rsf)$ the specialization map. By \cite[20.3]{fulton2016intersection}, the specialization map is induced by a map $\spmap:Z_{i}(Y_{\frf}/\frf)\rightarrow Z_{i}(Y_{\rsf}/\rsf)$.

\begin{thm}\label{specialization-comparison}
	Let $B=[b_{ij}]_{1\leq i\leq d;1\leq j\leq n}$ be a matrix whose entries are integers, and let $\chi=(\chi_{0},\dots,\chi_{d}):\mu_{q-1}^{d+1}\rightarrow\ql^{*}$ be a multiplicative character. If $B$ is sub-$p$-nondegenerate, then
	\begin{equation*}
		\spmap(CC\gkz_{\chi}(B))=CC\gkz_{\chi}(B)\in Z_{n}(T^{*}\afs^{n}/\rsf)
	\end{equation*}
	where we use $\gkz_{\chi}(B)$ to denote the $\ell$-adic GKZ-type sheaf over both $K$ and $\kappa$ respectively.
\end{thm}
\begin{proof}
	First by \Cref{resolution-existence}, there exists a complete generable set $\Sigma$ of cones in $\mathbb{Z}^{d}$, such that $B$ is $\Sigma$-good. By \Cref{divisor-good-is-normal}, we note that $\overline{D}_{B}(\Sigma)$ is a divisor with simple normal crossings. This means that $N_B(\Sigma)$ is a union of conormal bundles $T_D^*\prs^n_{\trv(\Sigma)}$ where $D$ runs though all intersections of subsets of the set of all irreducible components of the simple normal crossing divisor $\overline D_B (\Sigma)$. In fact, we note that here the $D_\Lambda$ is smooth over $\Lambda$ (with $\Lambda=\kappa, K$ or $R$). Therefore, we also know that $T_D^*\prs^n_{\trv(\Sigma)}$ is so. Then, the assertion follows from the argument in \cite{fulton2016intersection} between Proposition 20.3 and Corollary 20.3 on page 399. As a consequence, we have
	\begin{equation*}
		\spmap([N_{B}(\Sigma)])=[N_{B}(\Sigma)]\in Z_{d+n}(T^{*}\prs^{n}_{\trv(\Sigma)}/\rsf).
	\end{equation*}
	By {\cite[Proposition~20.3]{fulton2016intersection}}, we have the following commutative diagram.
	\begin{equation*}
	\begin{tikzcd}
		CH_{d+n}(N_{B}(\Sigma)/\frf)\arrow[d,"\overline{\pi}_{\Sigma!}"']\arrow[r,"\spmap"]&CH_{d+n}(N_{B}(\Sigma)/\rsf)\arrow[d,"\overline{\pi}_{\Sigma!}"]\\
		CH_{n}(\overline{\pi}_{\Sigma\circ}N_{B}(\Sigma)/\frf)\arrow[d,"j^{*}"']\arrow[r,"\spmap"]&CH_{n}(\overline{\pi}_{\Sigma\circ}N_{B}(\Sigma)/\rsf)\arrow[d,"j^{*}"]\\
		CH_{n}(j^{*}\overline{\pi}_{\Sigma\circ}N_{B}(\Sigma)/\frf)\arrow[r,"\spmap"]&CH_{n}(j^{*}\overline{\pi}_{\Sigma\circ}N_{B}(\Sigma)/\rsf)
	\end{tikzcd}.
	\end{equation*}
	As a consequence, we know that the equality
	\begin{equation}\label{equality_in_chow_group}
		\spmap(j^*\overline\pi_{\Sigma!}[N_B(\Sigma)])=j^*\overline\pi_{\Sigma!}\spmap([N_B(\Sigma)])
	\end{equation}
	holds in $CH_{n}(j^{*}\overline{\pi}_{\Sigma\circ}N_{B}(\Sigma)/\rsf)$. Moreover, by \Cref{image-direct-image-kummer-sheaf} and \Cref{dimension-global-dimension}, we have $\dim\overline{\pi}_{\Sigma\circ}N_{B}(\Sigma)=\dim j^{*}\overline{\pi}_{\Sigma\circ}N_{B}(\Sigma)=n$. Therefore, the equality \cref{equality_in_chow_group} also holds in the cycle group $Z_{n}(j^{*}\overline{\pi}_{\Sigma\circ}N_{B}(\Sigma)/\rsf)$ which can be identified with $Z_{n}(T^{*}\afs^{n}/\rsf)$. Then, the assertion follows from \Cref{square-characteristic-cycle}.
\end{proof}

\subsection{Comparison to GKZ hypergeometric $\mathcal{D}$-modules}\label{topology}

For a scheme $Y$ over $\frf$, we denote by $Y_{\mathbb{C}}=Y\times_{\spec\frf}\spec\mathbb{C}$. For a morphism $f:Y'\rightarrow Y$ of schemes over $\frf$, let $f_{\mathbb{C}}:Y'_{\mathbb{C}}\rightarrow Y_{\mathbb{C}}$ be the morphism induced by the base change $\spec\mathbb{C}\rightarrow\spec\frf$. For a matrix $A=[a_{ij}]_{0\leq i\leq d;1\leq j\leq n}$ whose entries are integers, let $\Phi(A)=\{\omega\in\mathbb{Z}^{n}\mid A\cdot w=0\}$. For $\omega=(\omega_{i})_{1\leq i\leq n}\in\Phi(A)$, we define
\begin{equation*}
	\square_{\omega}=\prod_{\omega_{i}>0}\partial_{x_{i}}^{\omega_{i}}-\prod_{\omega_{i}<0}\partial_{x_{i}}^{-\omega_{i}}.
\end{equation*}
Let $z=(z_{0},\dots,z_{d})\in(\mathbb{C}^{*})^{d+1}$. For $0\leq i\leq d$, we denote by
\begin{equation*}
	L_{i}(A)=\sum_{j=1}^{d}a_{ij}x_{j}\partial_{x_{j}}-z_{i}.
\end{equation*}

\begin{dfn}[{\cite[2.1]{gel1989hypergeometric}}]
	Let $A=[a_{ij}]_{0\leq i\leq d;1\leq j\leq n}$ be a non-confluent matrix, and let $z=(z_{0},\dots,z_{d})\in(\mathbb{C}^{*})^{d+1}$. The \textnormal{GKZ hypergeometric $\mathcal{D}$-module} associated to $A,z$ is defined to be
	\begin{equation*}
		\mathcal{H}_{z}(A)=\mathcal{D}_{\afs^{n}_{\mathbb{C}}}\biggm/\biggl(\sum_{i=0}^{d}\mathcal{D}_{\afs^{n}_{\mathbb{C}}}\cdot L_{i}(A)+\sum_{\omega\in\Phi(A)}\mathcal{D}_{\afs^{n}_{\mathbb{C}}}\cdot\square_{\omega}\biggr)\in D_{c}^{b}(\mathcal{D}_{\afs^{n}_{\mathbb{C}}}).
	\end{equation*}
\end{dfn}

For a subset $\theta\subset\{1,\dots,n\}$, let $\mathfrak{C}_{A}^{\theta}$ be the contraction of the ideal of $\frf[t_{0}^{\pm 1},\dots,t_{d}^{\pm 1}][\coor(0)]$ generated by
\begin{equation*}
	\Xi_{0}(A,\theta)\cup\biggl\{\sum_{j\in\theta}a_{ij}x_{j}\xi(x_{j})\biggm|0\leq i\leq n\biggr\}\cup\biggl\{\xi(x_{j})-\prod_{0\leq i\leq d}t_{i}^{a_{ij}}\biggm|j\in\theta\biggr\}
\end{equation*}
by the embedding $\frf[\coor(0)]\rightarrow\frf[t_{0}^{\pm 1},\dots,t_{d}^{\pm 1}][\coor(0)]$. Let $C_{A}^{\theta}$ be the closed subset of $T^{*}\afs^{n}$ defined by $\mathfrak{C}_{A}^{\theta}$. Let $\Theta_{A}$ be the set of subsets of $\{1,\dots,n\}$ defined by {\cite[Definition~3.7]{reichelt2021algebraic}} where it is called the '\textit{$A$-umbrella}' and is denoted by $\Phi_A$. Then, by {\cite[Theorem~2.1]{berkesch2020characteristic}}, for $\theta\in\Theta_{A}$, the closed subset $C_{A}^{\theta}$ of $T^{*}\afs^{n}$ is irreducible. For $\theta\in\Theta_{A}$, let $\mathrm{m}_{A,z}^{\theta}\in\mathbb{Z}$ be the \textit{Euler-Koszul $0$-characteristic} defined by {\cite[Definition~4.7]{schulze2008irregularity}} where it is denoted by $\mu_A^{L,\tau}$ with $L$ set to $0$ and $\tau$ replaced by $\theta$.

\begin{thm}[{\cite[II,~6.2,~Theorem]{hotta1998equivariant}},{\cite[Corollary~4.12]{schulze2008irregularity}}]\label{topology-gkz-result}
	Let $A=[a_{ij}]_{0\leq i\leq d;1\leq j\leq n}$ be a non-confluent matrix, and let $z=(z_{0},\dots,z_{d})\in(\mathbb{C}^{*})^{d+1}$. Then $\mathcal{H}_{z}(A)$ is regular holonomic and
	\begin{equation}\label{topology-gkz-formula}
		CC\mathcal{H}_{z}(A)=\sum_{\theta\in\Theta_{A}}\mathrm{m}_{A,z}^{\theta}[(C_{A}^{\theta})_{\mathbb{C}}]\in Z_{n}(T^{*}\afs^{n}_{\mathbb{C}}/\mathbb{C}).
	\end{equation}
\end{thm}
\begin{proof}
	We refer to {\cite[II,~6.2,~Theorem]{hotta1998equivariant}} for the proof of $\mathcal{H}_{z}(A)$ being regular holonomic, and refer to {\cite[Corollary~4.12]{schulze2008irregularity}} for the proof of \cref{topology-gkz-formula}.
\end{proof}

For a scheme $Y$ over $\frf$, we denote by $Y_{\topo}$ the topological space of $Y_{\mathbb{C}}$. By \cite[6.1.2]{beilinson2018faisceaux}, there is a fully faithful functor $\eta:D^{b}_{c}(Y_{\mathbb{C}},\ql)\rightarrow D_{c}^{b}(Y_{\topo},\ql)$ which commutes with the six functors. Fix an isomorphism $\iota:\ql\rightarrow\mathbb{C}$, then $\iota$ induces an equivalence of categories $D_{c}^{b}(Y_{\topo},\ql)\rightarrow D_{c}^{b}(Y_{\topo},\mathbb{C})$. We denote the composition of this equivalence functor and $\eta$ by
\begin{equation*}
	\eta_{\iota}:D^{b}_{c}(Y_{\mathbb{C}},\ql)\rightarrow D_{c}^{b}(Y_{\topo},\mathbb{C}).
\end{equation*}
For $\mathscr{F}\in D_{c}^{b}(Y,\ql)$, we denote by $\mathscr{F}_{\mathbb{C}}$ the inverse image along $Y_{\mathbb{C}}\rightarrow Y$ and by $\mathscr{F}_{\topo}=\eta_{\iota}(\mathscr{F}_{\mathbb{C}})$.

\begin{lem}\label{topology-gkz-isomorphism}
	Let $A=[a_{ij}]_{0\leq i\leq d;1\leq j\leq n}$ be a non-confluent matrix, and let $z=(z_{0},\dots,z_{d})\in(\mathbb{C}^{*})^{d+1}$. Then, we have
	\begin{equation*}
		\mathcal{H}_{z}(A)\cong R\pi_{\mathbb{C}!}(\tau_{\mathbb{C}}^{*}((\tau_{1})_{\mathbb{C}}^{*}\mathcal{K}_{z_{1}}\otimes\dots\otimes(\tau_{d})_{\mathbb{C}}^{*}\mathcal{K}_{z_{d}})\otimes(g_{B})_{\mathbb{C}}^{*}(j_{\dagger})_{\mathbb{C}!}\mathcal{K}_{-z_{0}}).
	\end{equation*}
	Here, for $c\in\mathbb{C}^{*}$, we denote by $\mathcal{K}_{c}=\mathcal{D}_{\spec\mathbb{C}[t]}/(\mathcal{D}_{\spec\mathbb{C}[t]}\cdot(t\partial_{t}-c))$.
\end{lem}
\begin{proof}
	Let $h_A:\mathbb T^{d+1}\rightarrow\mathbb A^n$ be the composition of the open immersion $\trs^n\hookrightarrow\afs^n$ and the morphism
	\begin{equation*}
		\trs^{d+1}\to\trs^n:(t_0,\dots,t_d)\mapsto(t_0^{a_{01}}\dots t_d^{a_{d1}},\dots,t_0^{a_{0n}}\dots t_d^{a_{dn}}).
	\end{equation*}
	By \cite[Theorem~2.10]{reichelt2021algebraic}, we know that $\mathcal H_z(A)$ is canonically isomorphic to the Fourier-Laplace transform of $(h_A)_{\mathbb C !}(\tau_{\mathbb{C}}^{*}((\tau_{0})_{\mathbb{C}}^{*}\mathcal{K}_{z_{0}}\otimes\dots\otimes(\tau_{d})_{\mathbb{C}}^{*}\mathcal{K}_{z_{d}})$. Then, the assertion follows from a same procedure as that in the proof of \cite[Lemma~1.1]{fu2016l} and \cite[\S~3]{fu2016l}.
\end{proof}

\begin{cor}\label{topology-comparison-easy}
	Let $B=[b_{ij}]_{0\leq i\leq d;1\leq j\leq n}$ be a sub-non-confluent matrix, and let $\chi=(\chi_{0},\dots,\chi_{d}):\mu_{q-1}^{d+1}\rightarrow\overline{\mathbb{Q}}_{\ell}^{*}$ be a multiplicative character. For $0\leq i\leq d$, set $z_{i}=\iota\chi_{i}(\zeta_{q-1})$. Write $z=(z_{0},\dots,z_{d})$, then the following statements are true.
	\begin{enumerate}
		\item $\gkz_{\chi}(B)_{\topo}\cong DR(\mathcal{H}_{z}(\widehat{B}))\in D_{c}^{b}(\afs^{n}_{\topo},\mathbb{C})$.
		\item $CC\mathscr{H}_{\chi}(\widehat{B})=(-1)^{d+n}\sum_{\theta\in\Theta_{\widehat{B}}}\mathrm{m}_{\widehat{B},z}^{\theta}[C_{\widehat{B}}^{\theta}]\in Z_{n}(T^{*}\afs^{n}/\frf)$.
	\end{enumerate}
\end{cor}
\begin{proof}
	For a multiplicative character $\rho:\mu_{q-1}\rightarrow\overline{\mathbb{Q}}_{\ell}^{*}$, we note that $(\kum_{\rho})_{\topo}\cong DR(\mathcal{K}_{\iota\rho(1)})$. Then, the first assertion follows from \Cref{sheaf-definition} and \Cref{topology-gkz-isomorphism}. Next, we prove 2. By \Cref{resolution-existence}, there exists a generable set $\Sigma$ of cones in $\mathbb{Z}^{d}$ which is complete, such that $B$ is $\Sigma$-good. By \Cref{divisor-good-is-normal}, we note that $\overline{D}_{B}(\Sigma)=\prs^{n}_{\trv(\Sigma)}\setminus U_{B}$ is a divisor with simple normal crossings. By \Cref{topology-gkz-isomorphism}, we have
	\begin{equation*}
		\mathcal{H}_{z}(B)\cong R(\overline{\pi}_{\Sigma})_{\mathbb{C}!}(\overline{j}_{B,\Sigma})_{\mathbb{C}!}\mathcal{K}_{z}(B)
	\end{equation*}
	where $\mathcal{K}_{z}(B)=\tau_{\mathbb{C}}^{*}((\tau_{1})_{\mathbb{C}}^{*}\mathcal{K}_{z_{1}}\otimes\dots\otimes(\tau_{d})_{\mathbb{C}}^{*}\mathcal{K}_{z_{d}})\otimes(g_{B})_{\mathbb{C}}^{*}(j_{\dagger})_{\mathbb{C}!}\mathcal{K}_{-z_{0}}$. We note that
	\begin{equation*}
		CC\mathcal{K}_{z}(B)=(-1)^{d+n}[N_{B}(\Sigma)_{\mathbb{C}}]\in Z_{d+n}(T^{*}(\afs^{n}_{\trv(\Sigma)})_{\mathbb{C}}/\mathbb{C}).
	\end{equation*}
	By {\cite[Proposition~9.4.2,~Proposition~9.4.3]{kashiwara2013sheaves}}, we have
	\begin{equation*}
		CC\mathcal{H}_{z}(B)=j_{\mathbb{C}}^{*}(\overline{\pi}_{\Sigma})_{\mathbb{C}!}CC\mathcal{K}_{z}(B)\in Z_{d+n}(T^{*}\afs^{n}_{\mathbb{C}}/\mathbb{C}).
	\end{equation*}
	Then, the assertion follows from \Cref{square-characteristic-cycle} and \Cref{topology-gkz-result}.
\end{proof}

\begin{thm}\label{topology-comparison-hard}
	Let $B=[b_{ij}]_{0\leq i\leq d;1\leq j\leq n}$ be a sub-non-confluent matrix. Let $\chi=(\chi_{0},\dots,\chi_{d}):\mu_{q-1}^{d+1}\rightarrow\overline{\mathbb{Q}}_{\ell}^{*}$ be a multiplicative character where $\chi_{0}$ is nontrivial, and let $\psi:\rsf\rightarrow\overline{\mathbb{Q}}_{\ell}^{*}$ be a nontrivial additive character. For $0\leq i\leq d$, set $z_{i}=\iota\chi_{i}(\zeta_{q-1})$. Write $z=(z_{0},\dots,z_{d})$. If $B$ is sub-$p$-nondegenerate, then we have
	\begin{equation*}
		CC\hyp_{\psi}(\widehat{B},\chi)=\sum_{\theta\in\Theta_{\widehat{B}}}\mathrm{m}_{\widehat{B},z}^{\theta}\spmap([C_{\widehat{B}}^{\theta}])\in Z_{n}(T^{*}\afs^{n}/\rsf).
	\end{equation*}
\end{thm}
\begin{proof}
	The assertion follows from \Cref{specialization-comparison} and \Cref{topology-comparison-easy}.
\end{proof}

\section{Examples of non-square matrices}\label{example}
\begin{exm}\label{example-p-nondegenerate}
	Let $B=\begin{bmatrix}0&0&1\end{bmatrix}$, then $\widehat{B}$ is non-square but is both non-confluent and $p$-nondegenerate. For a homogeneous polynomial $F\in\rsf[X_{0},X_{1},X_{2},X_{3}]$, we denote by $D_{F}$ the divisor of $\prs^{3}_{\rsf}$ defined by $F$. Then, we have the following result.
	\begin{center}
	\begin{tabular}{|c|ccccccccccccccc}
		\hline\rule{0pt}{1.2em}
		$\theta$&\multicolumn{3}{c|}{$\varnothing$}&\multicolumn{4}{c|}{$\{1\}$}&\multicolumn{4}{c|}{$\{2\}$}&\multicolumn{4}{c|}{$\{3\}$}\\[0.2em]
		\hline\rule{0pt}{1.5em}
		$S(\widehat{B},\theta)$&\multicolumn{3}{c|}{$T^{*}_{\prs^{3}}\prs^{3}$}&\multicolumn{4}{c|}{$T^{*}_{D_{X_{1}}}\prs^{3}\cup T^{*}_{\prs^{3}}\prs^{3}$}&\multicolumn{4}{c|}{$T^{*}_{D_{X_{2}}}\prs^{3}\cup T^{*}_{\prs^{3}}\prs^{3}$}&\multicolumn{4}{c|}{$T^{*}_{D_{X_{3}}}\prs^{3}\cup T^{*}_{\prs^{3}}\prs^{3}$}\\[0.5em]
		\hline\rule{0pt}{1.2em}
		$\theta$&\multicolumn{15}{c|}{$\{1,2\}$}\\[0.2em]
		\hline\rule{0pt}{1.5em}
		$S(\widehat{B},\theta)$&\multicolumn{15}{c|}{$T^{*}_{D_{X_{1}+X_{2}}}\prs^{3}\cup T^{*}_{\prs^{3}}\prs^{3}$}\\[0.5em]
		\hline\rule{0pt}{1.2em}
		$\theta$&\multicolumn{15}{c|}{$\{2,3\}$}\\[0.2em]
		\hline\rule{0pt}{1.5em}
		$S(\widehat{B},\theta)$&\multicolumn{15}{c|}{$T^{*}_{D_{X_{2}}\cap D_{X_{3}}}\prs^{3}\cup T^{*}_{D_{X_{2}}}\prs^{3}\cup T^{*}_{D_{X_{3}}}\prs^{3}\cup T^{*}_{\prs^{3}}\prs^{3}$}\\[0.5em]
		\hline\rule{0pt}{1.2em}
		$\theta$&\multicolumn{15}{c|}{$\{1,3\}$}\\[0.2em]
		\hline\rule{0pt}{1.5em}
		$S(\widehat{B},\theta)$&\multicolumn{15}{c|}{$T^{*}_{D_{X_{1}}\cap D_{X_{3}}}\prs^{3}\cup T^{*}_{D_{X_{1}}}\prs^{3}\cup T^{*}_{D_{X_{3}}}\prs^{3}\cup T^{*}_{\prs^{3}}\prs^{3}$}\\[0.5em]
		\hline\rule{0pt}{1.2em}
		$\theta$&\multicolumn{15}{c|}{$\{1,2,3\}$}\\[0.2em]
		\hline\rule{0pt}{1.5em}
		$S(\widehat{B},\theta)$&\multicolumn{15}{c|}{$T^{*}_{D_{X_{1}+X_{2}}\cap D_{X_{3}}}\prs^{3}\cup T^{*}_{D_{X_{1}+X_{2}}}\prs^{3}\cup T^{*}_{D_{X_{3}}}\prs^{3}\cup T^{*}_{\prs^{3}}\prs^{3}$}\\[0.5em]
		\hline
	\end{tabular}
	\end{center}
	This implies that $CC\hyp_{\psi}(\widehat{B},\chi)$ is given by \Cref{topology-comparison-hard} where $\chi=(\chi_{0},\chi_{1}):\mu_{q-1}^{2}\rightarrow\overline{\mathbb{Q}}_{\ell}^{*}$ is a multiplicative character with $\chi_{0}$ being non-trivial.
\end{exm}

\begin{rem}
	\Cref{example-p-nondegenerate} considers a case where though the non-confluent matrix is non-square, as long as it is still $p$-nondegenerate, the characteristic cycle of the corresponding $\ell$-adic GKZ hypergeometric sheaf is given by \Cref{topology-comparison-hard}.
\end{rem}

\begin{exm}\label{example-p-reducible}
	Let $B=\begin{bmatrix}0&p&2p\end{bmatrix}$ and we assume that $p>2$, then $\widehat{B}$ is non-square as well as non-confluent but is not $p$-nondegenerate. Then $S_{0}(\widehat{B},\{1,2\})$ is the closed subset of $T^{*}\afs^{3}$ defined by the ideal generated by $\{\xi(x_{3}),x_{1}\xi(x_{1})+x_{2}\xi(x_{2})\}$, which implies that
	\begin{equation*}
		\dim S(\widehat{B},\{1,2\})\geq \dim S_{0}(\widehat{B},\{1,2\})=4>3.
	\end{equation*}
	Let $\chi=(\chi_{0},\chi_{1}):\mu_{q-1}^{2}\rightarrow\overline{\mathbb{Q}}_{\ell}^{*}$ be a multiplicative character with $\chi_{0}$ being non-trivial, then in this case $CC\hyp_{\psi}(\widehat{B},\chi)$ may not be able to be given by \Cref{topology-comparison-hard}. However, we denote by $B^{\div}=\begin{bmatrix}0&1&2\end{bmatrix}$, then $\widehat{B^{\div p}}$ is non-confluent as well as $p$-nondegenerate and we have
	\begin{equation*}
		\hyp_{\psi}(\widehat{B},\chi)\cong\hyp_{\psi}(\widehat{B^{\div p}},\chi).
	\end{equation*}
	Therefore, the following result implies that $CC\hyp_{\psi}(\widehat{B},\chi)$ is given by \Cref{topology-comparison-hard}.
	\begin{center}
	\begin{tabular}{|c|ccccccccccccccc}
		\hline\rule{0pt}{1.2em}
		$\theta$&\multicolumn{3}{c|}{$\varnothing$}&\multicolumn{4}{c|}{$\{1\}$}&\multicolumn{4}{c|}{$\{2\}$}&\multicolumn{4}{c|}{$\{3\}$}\\[0.2em]
		\hline\rule{0pt}{1.5em}
		$S(\widehat{B},\theta)$&\multicolumn{3}{c|}{$T^{*}_{\prs^{3}}\prs^{3}$}&\multicolumn{4}{c|}{$T^{*}_{D_{X_{1}}}\prs^{3}\cup T^{*}_{\prs^{3}}\prs^{3}$}&\multicolumn{4}{c|}{$T^{*}_{D_{X_{2}}}\prs^{3}\cup T^{*}_{\prs^{3}}\prs^{3}$}&\multicolumn{4}{c|}{$T^{*}_{D_{X_{3}}}\prs^{3}\cup T^{*}_{\prs^{3}}\prs^{3}$}\\[0.5em]
		\hline\rule{0pt}{1.2em}
		$\theta$&\multicolumn{15}{c|}{$\{1,2\}$}\\[0.2em]
		\hline\rule{0pt}{1.5em}
		$S(\widehat{B},\theta)$&\multicolumn{15}{c|}{$T^{*}_{D_{X_{1}}\cap D_{X_{2}}}\prs^{3}\cup T^{*}_{D_{X_{1}}}\prs^{3}\cup T^{*}_{D_{X_{2}}}\prs^{3}\cup T^{*}_{\prs^{3}}\prs^{3}$}\\[0.5em]
		\hline\rule{0pt}{1.2em}
		$\theta$&\multicolumn{15}{c|}{$\{2,3\}$}\\[0.2em]
		\hline\rule{0pt}{1.5em}
		$S(\widehat{B},\theta)$&\multicolumn{15}{c|}{$T^{*}_{D_{X_{2}}\cap D_{X_{3}}}\prs^{3}\cup T^{*}_{D_{X_{2}}}\prs^{3}\cup T^{*}_{D_{X_{3}}}\prs^{3}\cup T^{*}_{\prs^{3}}\prs^{3}$}\\[0.5em]
		\hline\rule{0pt}{1.2em}
		$\theta$&\multicolumn{15}{c|}{$\{1,3\}$}\\[0.2em]
		\hline\rule{0pt}{1.5em}
		$S(\widehat{B},\theta)$&\multicolumn{15}{c|}{$T^{*}_{D_{X_{1}}\cap D_{X_{3}}}\prs^{3}\cup T^{*}_{D_{X_{1}}}\prs^{3}\cup T^{*}_{D_{X_{3}}}\prs^{3}\cup T^{*}_{\prs^{3}}\prs^{3}$}\\[0.5em]
		\hline\rule{0pt}{1.2em}
		$\theta$&\multicolumn{15}{c|}{$\{1,2,3\}$}\\[0.2em]
		\hline\rule{0pt}{1.5em}
		$\mathfrak{S}_{0}(\widehat{B},\theta)$&\multicolumn{15}{c|}{$\Bigl\langle 2\tfrac{X_{1}}{X_{0}}\xi\bigl(\tfrac{X_{1}}{X_{0}}\bigr)-\tfrac{X_{2}}{X_{0}}\xi\bigl(\tfrac{X_{2}}{X_{0}}\bigr),\tfrac{X_{1}}{X_{0}}\xi\bigl(\tfrac{X_{1}}{X_{0}}\bigr)-\tfrac{X_{3}}{X_{0}}\xi\bigl(\tfrac{X_{3}}{X_{0}}\bigr),\xi\bigl(\tfrac{X_{1}}{X_{0}}\bigr)\xi\bigl(\tfrac{X_{3}}{X_{0}}\bigr)-\xi\bigl(\tfrac{X_{2}}{X_{0}}\bigr)^{2}\Bigr\rangle$}\\[0.5em]
		\hline\rule{0pt}{1.5em}
		$\mathfrak{S}_{1}(\widehat{B},\theta)$&\multicolumn{15}{c|}{$\Bigl\langle\xi\bigl(\tfrac{X_{0}}{X_{1}}\bigr),\tfrac{X_{2}}{X_{1}}\xi\bigl(\tfrac{X_{2}}{X_{1}}\bigr)+2\tfrac{X_{3}}{X_{1}}\xi\bigl(\tfrac{X_{3}}{X_{1}}\bigr),\Bigl(\xi\bigl(\tfrac{X_{2}}{X_{1}}\bigr)+\xi\bigl(\tfrac{X_{3}}{X_{1}}\bigr)\Bigr)\xi\bigl(\tfrac{X_{3}}{X_{1}}\bigr)+\xi\bigl(\tfrac{X_{2}}{X_{1}}\bigr)^{2}\Bigr\rangle$}\\[0.5em]
		\hline\rule{0pt}{1.5em}
		$\mathfrak{S}_{2}(\widehat{B},\theta)$&\multicolumn{15}{c|}{$\Bigl\langle\xi\bigl(\tfrac{X_{0}}{X_{2}}\bigr),\tfrac{X_{1}}{X_{2}}\xi\bigl(\tfrac{X_{1}}{X_{2}}\bigr)-\tfrac{X_{3}}{X_{2}}\xi\bigl(\tfrac{X_{3}}{X_{2}}\bigr),\xi\bigl(\tfrac{X_{1}}{X_{2}}\bigr)\xi\bigl(\tfrac{X_{3}}{X_{2}}\bigr)-\Bigl(\xi\bigl(\tfrac{X_{1}}{X_{2}}\bigr)+\xi\bigl(\tfrac{X_{3}}{X_{2}}\bigr)\Bigr)^{2}\Bigr\rangle$}\\[0.5em]
		\hline\rule{0pt}{1.5em}
		$\mathfrak{S}_{3}(\widehat{B},\theta)$&\multicolumn{15}{c|}{$\Bigl\langle\xi\bigl(\tfrac{X_{0}}{X_{3}}\bigr),2\tfrac{X_{1}}{X_{3}}\xi\bigl(\tfrac{X_{1}}{X_{3}}\bigr)+\tfrac{X_{2}}{X_{3}}\xi\bigl(\tfrac{X_{2}}{X_{3}}\bigr),\xi\bigl(\tfrac{X_{1}}{X_{3}}\bigr)\Bigl(\xi\bigl(\tfrac{X_{1}}{X_{3}}\bigr)+\xi\bigl(\tfrac{X_{2}}{X_{3}}\bigr)\Bigr)+\xi\bigl(\tfrac{X_{2}}{X_{2}}\bigr)^{2}\Bigr\rangle$}\\[0.5em]
		\hline\rule{0pt}{1.5em}
		$\mathfrak{S}_{1}^{\infty}(\widehat{B},\theta)$&\multicolumn{15}{c|}{$\Bigl\langle\tfrac{X_{0}}{X_{1}},\tfrac{X_{2}}{X_{1}}\xi\bigl(\tfrac{X_{2}}{X_{1}}\bigr)+2\tfrac{X_{3}}{X_{1}}\xi\bigl(\tfrac{X_{3}}{X_{1}}\bigr),\Bigl(\xi\bigl(\tfrac{X_{2}}{X_{1}}\bigr)+\xi\bigl(\tfrac{X_{3}}{X_{1}}\bigr)\Bigr)\xi\bigl(\tfrac{X_{3}}{X_{1}}\bigr)+\xi\bigl(\tfrac{X_{2}}{X_{1}}\bigr)^{2}\Bigr\rangle$}\\[0.5em]
		\hline\rule{0pt}{1.5em}
		$\mathfrak{S}_{2}^{\infty}(\widehat{B},\theta)$&\multicolumn{15}{c|}{$\Bigl\langle\tfrac{X_{0}}{X_{2}},\tfrac{X_{1}}{X_{2}}\xi\bigl(\tfrac{X_{1}}{X_{2}}\bigr)-\tfrac{X_{3}}{X_{2}}\xi\bigl(\tfrac{X_{3}}{X_{2}}\bigr),\xi\bigl(\tfrac{X_{1}}{X_{2}}\bigr)\xi\bigl(\tfrac{X_{3}}{X_{2}}\bigr)-\Bigl(\xi\bigl(\tfrac{X_{1}}{X_{2}}\bigr)+\xi\bigl(\tfrac{X_{3}}{X_{2}}\bigr)\Bigr)^{2}\Bigr\rangle$}\\[0.5em]
		\hline\rule{0pt}{1.5em}
		$\mathfrak{S}_{3}^{\infty}(\widehat{B},\theta)$&\multicolumn{15}{c|}{$\Bigl\langle\tfrac{X_{0}}{X_{3}},2\tfrac{X_{1}}{X_{3}}\xi\bigl(\tfrac{X_{1}}{X_{3}}\bigr)+\tfrac{X_{2}}{X_{3}}\xi\bigl(\tfrac{X_{2}}{X_{3}}\bigr),\xi\bigl(\tfrac{X_{1}}{X_{3}}\bigr)\Bigl(\xi\bigl(\tfrac{X_{1}}{X_{3}}\bigr)+\xi\bigl(\tfrac{X_{2}}{X_{3}}\bigr)\Bigr)+\xi\bigl(\tfrac{X_{2}}{X_{2}}\bigr)^{2}\Bigr\rangle$}\\[0.5em]
		\hline
	\end{tabular}
	\end{center}
\end{exm}

\begin{rem}
	\Cref{example-p-reducible} considers a case where though the non-confluent matrix is neither square nor $p$-nondegenerate, there exists a matrix which is non-confluent as well as $p$-nondegenerate, such that the characteristic cycles of the corresponding $\ell$-adic GKZ hypergeometric sheaves are identical and given by \Cref{topology-comparison-hard}.
\end{rem}

\begin{exm}\label{example-p-degenerate}
	Let $B=\begin{bmatrix}0&1&p\end{bmatrix}$, then $\widehat{B}$ is non-square and non-confluent but not $p$-nondegenerate. Then $S_{0}(\widehat{B},\{1,3\})$ is the closed subset of $T^{*}\afs^{3}$ defined by the ideal generated by $\{\xi(x_{2}),x_{1}\xi(x_{1})+x_{3}\xi(x_{3})\}$, which implies that
	\begin{equation*}
		\dim S(\widehat{B},\{1,3\})\geq\dim S_{0}(\widehat{B},\{1,3\})=4>3.
	\end{equation*}
	Let $\chi=(\chi_{0},\chi_{1}):\mu_{q-1}^{2}\rightarrow\overline{\mathbb{Q}}_{\ell}^{*}$ be a multiplicative character with $\chi_{0}$ being non-trivial, then in this case $CC\hyp_{\psi}(\widehat{B},\chi)$ may not be able to be given by \Cref{topology-comparison-hard}.
\end{exm}

\clearpage
\phantomsection
\addcontentsline{toc}{section}{References}
\bibliographystyle{alpha}
\bibliography{reference.bib}

\begin{thebibliography}{BBDG18}

\bibitem[Abe22]{abe2022ramification}
Tomoyuki Abe.
\newblock Ramification theory from homotopical point of view, i.
\newblock {\em arXiv preprint arXiv:2206.02401}, 2022.

\bibitem[BBDG18]{beilinson2018faisceaux}
Alexander Beilinson, Joseph Bernstein, Pierre Deligne, and Ofer Gabber.
\newblock {\em Faisceaux pervers}.
\newblock Soci{\'e}t{\'e} math{\'e}matique de France Paris, 2018.

\bibitem[BFF20]{berkesch2020characteristic}
Christine Berkesch and Mar{\'\i}a-Cruz Fern{\'a}ndez-Fern{\'a}ndez.
\newblock Characteristic cycles and gevrey series solutions of a-hypergeometric systems.
\newblock {\em Algebra $\&$ Number Theory}, 14(2):323--347, 2020.

\bibitem[Fu16]{fu2016l}
Lei Fu.
\newblock \texorpdfstring{$\ell$}{l}-adic {GKZ} hypergeometric sheaves and exponential sums.
\newblock {\em Advances in mathematics}, 298:51--88, 2016.

\bibitem[Ful93]{fulton1993introduction}
William Fulton.
\newblock {\em Introduction to toric varieties}.
\newblock Number 131. Princeton university press, 1993.

\bibitem[Ful16]{fulton2016intersection}
William Fulton.
\newblock {\em Intersection theory}.
\newblock Princeton University Press, 2016.

\bibitem[GG01]{gelfand2001hypergeometric}
IM~Gelfand and MI~Graev.
\newblock Hypergeometric functions over finite fields.
\newblock In {\em Doklady Mathematics}, volume~64, pages 402--406. Pleiades Publishing, Ltd., 2001.

\bibitem[GZK89]{gel1989hypergeometric}
Izrail~Moiseevich Gel'fand, Andrei~Vladlenovich Zelevinskii, and Mikhail~Mikhailovich Kapranov.
\newblock Hypergeometric functions and toral manifolds.
\newblock {\em Functional Analysis and its applications}, 23(2):94--106, 1989.

\bibitem[Hot98]{hotta1998equivariant}
Ryoshi Hotta.
\newblock Equivariant \textit{D}-modules.
\newblock {\em arXiv preprint math/9805021}, 1998.

\bibitem[Kas03]{kashiwara2003d}
Masaki Kashiwara.
\newblock {\em D-modules and microlocal calculus}, volume 217.
\newblock American Mathematical Soc., 2003.

\bibitem[Kat88]{katz1988gauss}
Nicholas~M Katz.
\newblock {\em Gauss sums, Kloosterman sums, and monodromy groups}.
\newblock Number 116. Princeton university press, 1988.

\bibitem[Kat90]{katz1990exponential}
Nicholas~M Katz.
\newblock {\em Exponential sums and differential equations}.
\newblock Number 124. Princeton University Press, 1990.

\bibitem[KS13]{kashiwara2013sheaves}
Masaki Kashiwara and Pierre Schapira.
\newblock {\em Sheaves on Manifolds: With a Short History.Les d{\'e}buts de la th{\'e}orie des faisceaux. By Christian Houzel}, volume 292.
\newblock Springer Science \& Business Media, 2013.

\bibitem[RSSW21]{reichelt2021algebraic}
Thomas Reichelt, Mathias Schulze, Christian Sevenheck, and Uli Walther.
\newblock Algebraic aspects of hypergeometric differential equations.
\newblock {\em Beitr{\"a}ge zur Algebra und Geometrie/Contributions to Algebra and Geometry}, 62(1):137--203, 2021.

\bibitem[Sai17]{saito2017characteristic}
Takeshi Saito.
\newblock The characteristic cycle and the singular support of a constructible sheaf.
\newblock {\em Inventiones mathematicae}, 207(2):597--695, 2017.

\bibitem[Sai21]{saito2021characteristic}
Takeshi Saito.
\newblock Characteristic cycles and the conductor of direct image.
\newblock {\em Journal of the American Mathematical Society}, 34(2):369--410, 2021.

\bibitem[SW08]{schulze2008irregularity}
Mathias Schulze and Uli Walther.
\newblock Irregularity of hypergeometric systems via slopes along coordinate subspaces.
\newblock {\em Duke Mathematical Journal}, 142(3):465--509, 2008.

\end{thebibliography}

\end{document}